\documentclass[sn-mathphys,Numbered]{sn-jnl}

\usepackage{subcaption}
\usepackage{graphicx}%
\usepackage{multirow}%
\usepackage{amsmath,amssymb,amsfonts}%
\usepackage{amsthm}%
\usepackage{mathrsfs}%
\usepackage[title]{appendix}%
\usepackage{algorithm}
\usepackage[noend]{algpseudocode}
\usepackage{xcolor}%
\usepackage{textcomp}%
\usepackage{manyfoot}%
\usepackage{booktabs}%
\usepackage{enumitem}
\usepackage{bbm}
\usepackage{mathptmx}
\usepackage{ragged2e,booktabs,tabularx,adjustbox}

\newcommand{\scrF}{\mathscr{F}}

\newtheorem{theorem}{Theorem}
\begin{document}

\title[Article Title]{Self-Decomposable Laws Associated with  General Tempered Stable (GTS) Distribution and their Simulation Applications}


\author*[1]{\fnm{Aubain} \sur{Nzokem}}\email{hilaire77@gmail.com}





\abstract{The paper describes the self-decomposable distribution and the background driving L\'evy process (BDLP)  associated with the Generalized Tempered Stable (GTS) distribution. Two distributions are provided: the background driving L\'evy process (BDLP) of the GTS distribution and the self-decomposable distribution generated by the GTS distribution as BDLP. The derived self-decomposable distribution and the GTS distribution are used as stationary distribution in the Ornstein-Uhlenbeck type process. A simulation method, based on sampling the random integral representation, is applied to mimic the S\&P 500  Index and Bitcoin daily cumulative return process.}




\maketitle
 \section{Introduction}\label{sec1}
 \noindent
\noindent
The Tempered stable distribution was introduced to overcome the limitations of the stable distribution and the normal distribution. It was shown \cite{rosinski2007tempering} that The tempered stable process behaves in the short times as a stable process and in the long term as a Gaussian process. The tails of the Tempered stable distribution are heavier than the normal distribution and thinner than the stable distribution \cite{kim2008financial,kim2009new}. Tempered stable models are used in several areas of physics \cite{ JanRosinski2010}, inclusing turbulence, plasma physic , solar winds, and their popularity is still growing in financial applications.\\
Although the spectral measure\cite{rosinski2007tempering, kim2008financial, bianchi2011tempered} provide the analytical properties of all classes of Tempered stable distributions,  the Generalised tempered stable (GTS) process, which is the most popular subclass, is constructed \cite{rachev2011financial,kuchler2013tempered} by tempering the L\'evy measure of a stable distribution. The six-parameter L\'evy measure of the GTS distribution (\ref{eq:l3}) is a product of a tempering or tilting function ($q(y)$) (\ref{eq:l1}) and a L\'evy measure of the $\alpha$-stable distribution ($M_{stable}(dy)$) (\ref{eq:l2})  described as follows.
 \begin{align}
 q(y) &= e^{-\lambda_{+}y} \boldsymbol{1}_{(0,\infty)}(y)+ e^{-\lambda_{-}|y|} \boldsymbol{1}_{(-\infty, 0)}(y) \label{eq:l1}\\
M_{stable}(dy) &=\left(\frac{\alpha_{+}}{y^{1+\beta_{+}}} \boldsymbol{1}_{(0,\infty)}(y)+ \frac{\alpha_{-}}{|y|^{1+\beta_{-}}} \boldsymbol{1}_{(-\infty, 0)}(y)\right) dy \label{eq:l2}\\
M(dy) =q(y)M_{stable}(dy)&=\left[\frac{\alpha_{+}}{y^{1+\beta_{+}}} e^{-\lambda_{+}y}\boldsymbol{1}_{(0,\infty)}(y)\right] dy  + \left[\frac{\alpha_{-}}{|y|^{1+\beta_{-}}}e^{-\lambda_{-}|y|} \boldsymbol{1}_{(-\infty, 0)}(y)\right] dy \label{eq:l3}
 \end{align}
Where $0\leq \beta_{+}\leq 1$, $0\leq \beta_{-}\leq 1$, $\alpha_{+}\geq 0$, $\alpha_{-}\geq 0$, $\lambda_{+}\geq 0$ and $\lambda_{-}\geq 0$. \\

\noindent
The GTS distribution in (\ref{eq:l3}) is a six-parameter family of infinitely divisible distribution, which covers several well-known distribution subclasses: the KoBol distribution \cite{boyarchenko2002non} is obtained by substituting $\beta=\beta_{+}=\beta_{-}$, the truncated L\'evy flight \cite{kuchler2013tempered} by substituting $\lambda=\lambda_{+}= \lambda_{-} $ and $\beta=\beta_{+}=\beta_{-}$, the bilateral Gamma distributions \cite{KUCHLER2008261} by substituting $\beta_{+}=\beta_{-}=0$, the Variance Gamma distributions \cite{madan1998variance, nzokem2022,nzokem_2021b,nzokem2021fitting, Nzokem_Montshiwa_2023} by substituting $\alpha= \alpha_{+}= \alpha_{-}$ and $\beta_{+}=\beta_{-}=0$ while the Carr-Geman-Madan-Yor (CGMY) distributions \cite{carr2003stochastic}, also called Classical Tempered Stable Distribution\cite{rachev2011financial} is obtained by substituting  $\alpha= \alpha_{+}= \alpha_{-}$, $G=\lambda_{-}$, $M=\lambda_{+}$ and $Y=\beta_{+}=\beta_{-}$.\\
\noindent
The goal of the present paper is to investigate the self-decomposable distributions associated with  General tempered stable (GTS) distribution and provide a simulation technique to sample  the associated Ornstein-Uhlenbeck type processes.\\
There are different methods for simulating the Tempered Sable processes, but the choice usually depends on the structure of the L\'evy measures. For example, the CGMY process \cite{madan2006cgmy}, also called the Classical Tempered Stable process, can be simulated as a time-changed Brownian motion with drift. The method with the inverse of the L\'evy measure, in the series representations \cite{Rosinski2001}, was not suitable to implement in our case. The tail of  L\'evy measures does not have an explicit inverse.  The simulation method used in the paper is based on the Ornstein-Uhlenbeck type process, and the classic sampling method of inversion was applied to the random integral representation. The inverse of the cumulative probability function was provided by the composite FRFTs algorithm \cite{nzokem2023enhanced}.\\

\noindent
The remainder of the text is organized as follows: In section 2, we provide an overview of the framework for L\'evy process, Generalized Tempered Stable (GTS) Process, Ornstein-Uhlenbeck process, and self-decomposable distribution. In section 3, we investigate the background-driving L\'evy process (BDLP) of the GTS distribution and the self-decomposition distribution generated by the GTS distribution as a background-driving l\'evy process (BDLP). Section 4 presents the GTS Parameter Estimation and analyses the l\'evy density and characteristic functions derived in section 3. In section 4, we develop a simulation method and apply to sample the S\&P 500  Index and Bitcoin daily cumulative return process.
\newpage
\section {Preliminaries}
\noindent
\subsection{L\'evy Framework and Asset Pricing: Overview}
\noindent
Let $(\Omega, \scrF, \left\{ \scrF_{t} \right \}_{t \geq 0},  \mathbf{P})$ be a filtered probability space, with $\scrF =\vee_{t>0}\scrF_{t}$ and $\left\{ \scrF_{t} \right \}_{t \geq 0}$ is a filtration.  $\scrF_{t}$ is a $\sigma$-algebra included in $\scrF$ and for $\tau<t$, $\scrF_{\tau}\subseteq \scrF_{t}$.\\
\noindent
A stochastic process $L=\{L_{t}\}_{t\geq0}$ is a L\'evy process, if it has the following properties\\
(A1): $L_{0} = 0$ a.s;\\
(A2): $L_{t }$ has independent increments, that is,  for $0<t_{1} < t_{2} <\dots< t_{n}$, the random variables $L_{t_{1} } - L_{0} $, $L_{t_{2} } - L_{t_{1} }$, \dots, $L_{t_{n} } - L_{t_{n-1} }$  are independent;\\
(A3): $L_{t }$ has stationary increments, that is, 
 for any $t_{1} < t_{2} < +\infty$, the probability distribution  of $L_{t_{2} } - L_{t_{1} }$   depends only on $t_{2}  - t_{1} $ ;\\
(A4): $L_{t }$ is stochastically continuous: for any $t$ and $\epsilon>0$, $\lim_{s \to t} P\left(|L_{s} - L_{t}|>\epsilon \right)=0$;\\
(A5): $c\grave{a}dl\grave{a}g$ paths, that is, $t \mapsto L_{t}$ is a.s right continuous with left limits\\

\noindent
Given a L\'evy process $L=\{L_{t}\}_{t\geq0}$  on the filtered probability space $(\Omega, \scrF, \left\{ \scrF_{t} \right \}_{t \geq 0}, \mathbf{P})$, we define the asset value process $S=\{S_{t}\}_{t\geq0}$ as $S_{t}=S_{0}e^{L_{t}}$.\\

 \begin{theorem}\label{lem1} (L\'evy \& Khintchine representation) \ \\
 Let $L=\{L_{t}\}_{t\geq0}$ be a L\'evy precess on $\mathbb{R}$ with characteristic triplet $(\gamma, \sigma^{2}, M)$. Then for any $\xi \in \mathbb{R}$ and $t \geq 0$, we have:
  \begin{equation}
 \begin{aligned}
E\left[e^{i L_{t}\xi}\right]&=e^{t\varphi(\xi)}\\
\varphi(\xi)= Log\left(Ee^{i L_{1}\xi}\right)&= i\gamma \xi - \frac{1}{2} \sigma^{2}\xi^{2} + \int_{-\infty}^{+\infty}\left(e^{i \xi y} -1 - y\xi 1_{|y|\leq 1}\right)M(dy)
\label {eq:l02}
  \end{aligned}
\end{equation}
\noindent 
Where $\gamma \in \mathbb{R}$,  $\sigma \geq 0$ is the Gaussian coefficient, $\varphi(\xi)$ is the characteristic exponent and $M$ is the L\'evy measure of L such that
  \begin{equation*}
 \begin{aligned}
 \int_{-\infty}^{+\infty}Min(1,|y|^{2})M(dy) <+\infty
\label {eq:l021}
  \end{aligned}
\end{equation*} 
\end{theorem} 

\noindent
For the theorem-proof, see \cite{applebaum2009levy,ken1999levy,tankov2003financial}\\

\noindent
The L\'evy process is completely determined by the L\'evy–Khintchine triplet  $(\gamma, \sigma^{2}, M)$. The triplet terms suggest that the L\'evy process has three independent components: a linear drift, a Brownian motion, and a L\'evy jump process. When the diffusion term $\sigma=0$, we have a  L\'evy jump process; in addition, if $\gamma=0$, we have a pure jump process.\\

\noindent
The L\'evy-Khintchine representation has a simpler expression when the sample paths of the L\'evy process have finite variation on every compact time interval almost surely. A L\'evy process has bounded variation\cite{Itkin2017} if and only if
\begin{equation}
\begin{aligned}
\sigma^{2}=0  \quad \quad \mbox{ and }  \quad \quad \int_{-\infty}^{+\infty} min(1,|x|)M(dx) <+ \infty  \label {eq:l283}
  \end{aligned}
\end{equation} 

\noindent
 With the assumption of L\'evy process with finite variation, the characteristic exponent $\varphi(\xi)$ in the theorem \ref{lem1} becomes
   \begin{equation}
 \begin{aligned}
\varphi(\xi)= Log\left(Ee^{i L_{1}\xi}\right)= i\Gamma \xi  + \int_{-\infty}^{+\infty}\left(e^{i \xi y} -1 \right)M(dy)
  \end{aligned}
\end{equation} 
With
\begin{equation*}
 \begin{aligned}
\sigma^{2}=0 \quad \quad \Gamma=\gamma + \int_{-1}^{1} x M(dx) 
  \end{aligned}
\end{equation*} 
 
\subsection { Generalized Tempered Stable (GTS) Process: Overview}
\noindent
The L\'evy measure of the Generalized Tempered Stable (GTS) distribution ($M(dy)$) is defined (\ref{eq:l23}) as a product of a tempering or tilting \cite{bianchi2009tempered,rosinski2007tempering} function ($q(y)$) (\ref{eq:l21}) and a L\'evy measure of the $\alpha$-stable distribution ($M_{stable}(dy)$) (\ref{eq:l22}).
 \begin{align}
q(y) &= e^{-\lambda_{+}y} \boldsymbol{1}_{y>0} + e^{-\lambda_{-}|y|} \boldsymbol{1}_{y<0} \label{eq:l21}\\
M_{stable}(dy) &=\left(\frac{\alpha_{+}}{y^{1+\beta_{+}}} \boldsymbol{1}_{y>0} + \frac{\alpha_{-}}{|y|^{1+\beta_{-}}} \boldsymbol{1}_{y<0}\right) dy \label{eq:l22}\\
M(dy) =q(y)M_{stable}(dy)&=\left(\frac{\alpha_{+}e^{-\lambda_{+}y}}{y^{1+\beta_{+}}} \boldsymbol{1}_{y>0} + \frac{\alpha_{-}e^{-\lambda_{-}|y|}}{|y|^{1+\beta_{-}}} \boldsymbol{1}_{y<0}\right) dy \label{eq:l23}
 \end{align}
\noindent
Where $0\leq \beta_{+}\leq 1$, $0\leq \beta_{-}\leq 1$, $\alpha_{+}\geq 0$, $\alpha_{-}\geq 0$, $\lambda_{+}\geq 0$ and $\lambda_{-}\geq 0$. \\

\noindent
The six parameters that appear have important interpretations. $\beta_{+}$ and $\beta_{-}$ are the indexes of stability bounded below by 0 and above by 2 \cite{borak2005stable}. They capture the peakedness of the distribution similarly to the $\beta$-stable distribution, but the distribution tails are tempered. If $\beta$ increases (decreases), then the peakedness decreases (increases). $\alpha_{+}$ and $\alpha_{-}$ are the scale parameters, also called the process intensity \cite{boyarchenko2002non}; they determine the arrival rate of jumps for a given size. $\lambda_{+}$ and $\lambda_{-}$ control the decay rate on the positive and negative tails. Additionally, $\lambda_{+}$ and $\lambda_{-}$ are also skewness parameters. If $\lambda_{+}>\lambda_{-}$ ($\lambda_{+}<\lambda_{-}$), then the distribution is skewed to the left (right), and if $\lambda_{+}=\lambda_{-}$, then it is symmetric \cite{rachev2011financial, fallahgoul2017quantile}. $\alpha$ and $\lambda$ are related to the degree of peakedness and thickness of the distribution. If $\alpha$ increases (decreases), the peakedness and the thickness decrease (increase). Similarly, If $\lambda$ increases (decreases), then the peakedness increases (decreases) and the thickness decreases (increases) \cite{bianchi2019handbook}.\\

\noindent
The GTS distribution can be denoted by $X\sim GTS(\textbf{$\beta_{+}$}, \textbf{$\beta_{-}$}, \textbf{$\alpha_{+}$},\textbf{$\alpha_{-}$}, \textbf{$\lambda_{+}$}, \textbf{$\lambda_{-}$})$ and $L_{1} =L^{+}_{1} - L^{-}_{1}$ with $L^{+}_{1} \geq 0$, $L^{-}_{1} \geq 0$. $L^{+}_{1}\sim TS(\textbf{$\beta_{+}$}, \textbf{$\alpha_{+}$},\textbf{$\lambda_{+}$})$ and $L^{-}_{1}\sim TS(\textbf{$\beta_{-}$}, \textbf{$\alpha_{-}$},\textbf{$\lambda_{-}$})$. \\

\noindent
The activity process of the GTS distribution can be studied from the integral (\ref{eq:l24}) of the L\'evy measure (\ref{eq:l23}). As shown in (\ref{eq:l24}), if $\beta_{+} < 0$, TS(\textbf{$\beta_{+}$}, \textbf{$\alpha_{+}$},\textbf{$\lambda_{+}$}) is of finite activity process and can be written as a compound Poisson. We have a similar pattern when $\beta_{-} < 0$. The interesting case is when $\beta_{+} \ge 0$. As shown in (\ref{eq:l24}), if $\beta_{+} \ge 0$, TS(\textbf{$\beta_{+}$}, \textbf{$\alpha_{+}$},\textbf{$\lambda_{+}$}) is of infinite activity process with infinite jumps in any given time interval. We have a similar pattern when $0 \le \beta_{-}$:
\begin{align}
 \int_{-\infty}^{+\infty} M(dy) =\begin{cases}
  +\infty  &\text{if }{\beta_{+}\ge 0\vee\beta_{-} \ge 0}   \\
  \alpha_{+}{\lambda_{+}}^{\beta_{+}}\Gamma(-\beta_{+},0) +  \alpha_{-}{\lambda_{-}}^{\beta_{-}}\Gamma(-\beta_{-},0) &\text{if }{\beta_{+} < 0\wedge\beta_{-} < 0} \end{cases} \label{eq:l24}
     \end{align}
where $\Gamma(s,x)=\int_{x}^{+\infty}y^{s-1}e^{-y}dy$ is the upper incomplete gamma function. In addition, the variation or smoothness of the process can be studied through the following integral:
 \begin{equation*}
 \begin{aligned}
&  \int_{-\infty}^{+\infty} min(1,|y|)M(dy)=   \int_{-\infty}^{-1} M(dy) +\int_{-1}^{0} |y|M(dy) + \int_{1}^{+\infty} M(dy) + \int_{0}^{1} |y|M(dy) \\
  &=\alpha_{-}\lambda_{-}^{\beta_{-}}\left[\Gamma(-\beta_{-},\lambda_{-}) + \lambda_{-}^{-1}\gamma(1-\beta_{-},\lambda_{-})\right] + \alpha_{+}\lambda_{+}^{\beta_{+}}\left[\Gamma(-\beta_{+},\lambda_{+}) + \lambda_{+}^{-1}\gamma(1-\beta_{+},\lambda_{+})\right]
 \end{aligned}
\end{equation*}
\noindent
where $\gamma(s,x)=\int_{0}^{x}y^{s-1}e^{-y}dy$ is lower incomplete gamma function.\\
We have
\begin{align}
  \int_{-\infty}^{+\infty} min(1,|y|)M(dy) <+ \infty  \hspace{5mm}
   \hbox{ if $0\leq \beta_{-}<1$ \& $0\leq \beta_{+}\leq 1$}\label{eq:l25}
\end{align}
\noindent
As shown in (\ref{eq:l25}), GTS(\textbf{$\beta_{+}$}, \textbf{$\beta_{-}$}, \textbf{$\alpha_{+}$},\textbf{$\alpha_{-}$}, \textbf{$\lambda_{+}$}, \textbf{$\lambda_{-}$}) is a finite variance process, which is not the case for the Brownian motion process.\\

\noindent
By adding the location parameter, we have:
 \begin{align}
L =\mu + L_{1}=\mu + L^{+}_{1} - L^{-}_{1} \quad \quad  L\sim GTS(\mu, \beta_{+}, \beta_{-}, \alpha_{+}, \alpha_{-},\lambda_{+}, \lambda_{-}) \label {eq:l26}
  \end{align}
The GTS distribution in (\ref{eq:l26}) is a seven-parameter family of infinitely divisible distribution, which covers several well-known distribution subclasses: the KoBol distribution \cite{boyarchenko2002non} is obtained by substituting $\beta=\beta_{+}=\beta_{-}$, the truncated L\'evy flight \cite{kuchler2013tempered} by substituting $\lambda=\lambda_{+}= \lambda_{-} $ and $\beta=\beta_{+}=\beta_{-}$, the bilateral Gamma distributions \cite{KUCHLER2008261} by substituting $\beta_{+}=\beta_{-}=0$, the Variance Gamma distributions \cite{madan1998variance, nzokem2022,nzokem_2021b,nzokem2021fitting, Nzokem_Montshiwa_2023} by substituting $\alpha= \alpha_{+}= \alpha_{-}$ and $\beta_{+}=\beta_{-}=0$ while the Carr-Geman-Madan-Yor (CGMY) distributions \cite{carr2003stochastic}, also called Classical Tempered Stable Distribution\cite{rachev2011financial} is obtained by substituting  $\alpha= \alpha_{+}= \alpha_{-}$, $G=\lambda_{-}$, $M=\lambda_{+}$ and $Y=\beta_{+}=\beta_{-}$.\\

\begin{theorem}\label{lem2} \ \\
Consider a variable $Y \sim GTS(\textbf{$\mu$}, \textbf{$\beta_{+}$}, \textbf{$\beta_{-}$}, \textbf{$\alpha_{+}$},\textbf{$\alpha_{-}$}, \textbf{$\lambda_{+}$}, \textbf{$\lambda_{-}$})$, the characteristic exponents can be written as
 \begin{align}
\Psi(\xi)&=Log\left(Ee^{i L\xi}\right)=\mu\xi i + \Psi^{+}(\xi) + \Psi^{-}(-\xi)\label {eq:l271}
 \end{align}
 Where 
  \begin{align*}
\Psi^{+}(\xi) &=Log\left(Ee^{i L^{+}_{1}\xi}\right)= \alpha_{+}\Gamma(-\beta_{+})\left((\lambda_{+} - i\xi)^{\beta_{+}} - {\lambda_{+}}^{\beta_{+}}\right) \\
\Psi^{-}(\xi) &=Log\left(Ee^{i L^{-}_{1}\xi}\right)=\alpha_{-}\Gamma(-\beta_{-})\left((\lambda_{-} - i\xi)^{\beta_{-}} - {\lambda_{-}}^{\beta_{-}}\right)
 \end{align*}
\end{theorem}

See \cite{nzokem2022fitting, nzokem2023european, nzokem2023bitcoin} for theorem \ref{lem2} proof. 


\subsection{Decomposable Laws: Overview}
\noindent
Consider a sequence $(X_{k} : k = 1, 2,...)$ of independent random variables, but not necessarily identically distributed;  and another sequence $(S_{n} : n = 1, 2,...)$ with $S_{n} = \sum_{k=0}^{n}X_{k}$. We assumed there exist a sequence $(c_{n} : n = 1, 2,...)$ and a scaling sequence $(b_{n} : n = 1, 2,...)$ such that  $\frac{S_{n} - c_{n}}{b_{n}}$ converges in distribution to some random variable $X$. In other words, the random variable $X$ distribution is the limiting distribution of a sequence of normalized sums of independent but not necessarily identically distributed random variables. The class of limit laws is known as L\'evy class L distributions or class L in the literature\cite{carr2007self}, and the random variable $X$ is said to belong to Class L. In the classical Central Limit Theorem (CLT), the asymptotically standard normal is of Class L. Likewise, the class of stable laws is of Class L. It can be described by four parameters:  a location parameter ($\mu$), a scale parameter ($\sigma > 0$), a stability index  ( $0 \leq \alpha \leq 2$) and a skewness parameter ($ -1 \leq \beta \leq 1$). The distribution of any random variable in the class L is infinitely divisible.\\

\noindent
 In addition, it is shown \cite{ken1999levy} that a Class L is also a class of self-decomposable probability distributions. The distribution of a random variable $X$ is  self-decomposable if for any constant $c$ ($ 0 < c < 1$ ), there exists an independent random variable $X_{c}$ such that
\begin{equation}
\begin{aligned}
X \overset{d}{=}  cX + X_{c}  \hspace{10mm}  \hbox{with $X$ and  $X_{c}$ are  independent variable} \label {eq:l281}  \end{aligned}
\end{equation} 
where $\overset{d}{=}$  means equality in distribution.\\
\noindent
Analytically,  a random variable $X$ with a characteristic function $\phi$ is self-decomposable if
\begin{equation}
\begin{aligned}
\forall (0 < c < 1) \,\exists \rho_{c} \, \forall (t \in \mathbb{R})  \quad \quad \phi(t)=\phi(ct)\rho_{c}(t) \label {eq:l282}  
\end{aligned}
\end{equation} 
where $\rho_{c}$ is also a characteristic function.\\

\noindent
Class L distributions is also described \cite{jurek1996series} as distributions of random integrals. Class L is represented as follows
\begin{equation}
\begin{aligned}
L=\left\{ \int_{0}^{+\infty} e^{-t} dY(s), \mbox{with $Y$ is a L\'evy Process and  $ E\left[\log(1+|Y(1)|)\right] \leq \infty $ } \right\} \label {eq:l283}
  \end{aligned}
\end{equation} 
where the L\'evy process $Y$ is called the background driving process of a self-composable distribution. In terms of characteristic functions, the random integral representation provides another characteristic of the Class L \cite{jurek1996series}.
\begin{equation}
\begin{aligned}
X \in   \mbox{Class L \quad if only if}  \quad \log(\phi(\xi))=\int_{0}^{\xi} \log(\varphi(u))\frac{du}{u} \label {eq:l284} 
 \end{aligned}
\end{equation} 
where $\phi(x)$ and $\varphi(x)$ are characteristic functions of $X$ and $Y (1)$\\

\noindent
The relationship between a self-decomposable variable $(X)$ and its background driving process can also be featured\cite{Jurek2023,jurek1996series} by the triplet in L\'evy-Khintchine formula (\ref{eq:l02}).  Let $\left[ A_{X}, \sigma^{2}_{X}, M_{X}\right]$ and  $\left[ A_{Y}, \sigma^{2}_{Y}, M_{Y}\right]$ are triplets in the L\'evy-Khintchine formula  for $X$ and $Y$ respectively. we have 
\begin{equation}
\begin{aligned}
M_{X}(A)&=\int_{0}^{+\infty}M_{Y}(e^{t}A)dt \quad  \mbox{ for borel subsets} \quad A \subset \mathbb{R} \setminus {\{0\}} \label {eq:l2851}\\
\sigma_{X}^{2}&=\frac{1}{2}\sigma_{Y}^{2}  \quad \quad A_{X}=A_{Y} + \int_{-\infty}^{+\infty} \left(\frac{x}{1+x^2} - arctan(x)\right )M_{Y}(dx) 
 \end{aligned}
\end{equation} 

\noindent
The class of self-decomposable probability distributions is large and includes among others\cite{jurek2022background}: Stable, Gamma, Log-Gamma, Log-Normal, Chi-Squared ($\chi^{2}$), Student-T, Log-Student-T, Fisher–Snedecor and Hyperbolic distributions.

\subsection{Ornstein-Uhlenbeck process: Overview} 
\noindent
The Ornstein-Uhlenbeck process is a diffusion process introduced by Ornstein and Uhlenbeck \cite{uhlenback1930theory} to model the stochastic behavior of the velocity of a particle undergoing Brownian motion. The Ornstein-Uhlenbeck diffusion $X=\{X(t)\}_{t\geq 0}$ is the solution of the Langevin Stochastic Differential Equation (SDE) (\ref{eq:l01})
\begin{align}
dX_{t}=-\lambda X_{t}dt + dB_{t} .
\label {eq:l01}
  \end{align}
where $\lambda >0$ and $B=\{B_{t},t\geq 0\}$ is a Brownian motion.\\

\noindent
The Ornstein-Uhlenbeck type process is used in finance to capture important distributional deviations from Gaussianity and to model dependence structures \cite{barndorff1998some, barndorff2000non, barndorff2001non, barndorff2003integrated}.  The extension of the Ornstein-Uhlenbeck processes is possible by replacing the Brownian motion in (\ref{eq:l01}) by the L\'evy process $\{L_{t}\}_{t\geq0}$, called the background driving L\'evy process (BDLP) corresponding to the process $\{X_{t}\}_{t\geq0}$. The SDE (\ref{eq:l01}) becomes
\begin{align} 
dX_{t}=-\lambda X_{t}dt + dL_{\lambda t}  \quad  \lambda >0.
\label {eq:l02}
  \end{align}

\noindent
$dL(\lambda t)$ is deliberately chosen \cite{barndorff2000non} so that the marginal distribution of X is unchanged, whatever the value of $\lambda$. The solution to the SDE (\ref{eq:l02})  is a stationary process $\{X_{t}\}_{t\geq0}$ such that $X_{t} \overset{d}{=} X$. $\{X_{t}\}_{t\geq0}$ is a stationary process of Ornstein-Uhlenbeck type developed by Barndorff-Nielsen and Shephard \cite{barndorff2000econometric,barndorff2003integrated,barndorff2001modelling}.  $\{X_{t}\}_{t\geq0}$ moves up entirely
by jumps and then tails off exponentially \cite{barndorff2000non}.
\newpage
\begin{theorem} \label{lem2} \ \\
Let $\phi$ be the characteristic function of a random variable X . If X is self-decomposable then there exists a stationary stochastic process 
$\left\{X(t)\right\}_{t>0}$ and a background driving L\'evy process (BDLP) $\left\{L(t)\right\}_{t>0}$ independent of $X(0)$ , such that $X(t)\overset{d}{=}X$ and 
\begin{align}
X(t)=e^{-\lambda t} X(0) + \int_{0}^{t}e^{-\lambda (t-s)} dL_{\lambda s}    \quad  \lambda >0. 
\label {eq:l03}
  \end{align}
Conversely, if $\left\{X(t)\right\}_{t>0}$ is a stationary stochastic process and  $\left\{L(t)\right\}_{t>0}$ is a L\'evy process, independent of $X(0)$, such that $\left\{X(t)\right\}_{t>0}$ and   $\left\{L(t)\right\}_{t>0}$ satisfy the equation (\ref{eq:l02}) for all  $\lambda >0$ then $X(t)$  is
self-decomposable.
\end{theorem} 

For the theorem-proof, see \cite{barndorff2001modelling}
\section { Self-decomposability of GTS distribution }

\subsection{ Background Driving L\'evy Process (BDLP) of the GTS distribution}

\begin{theorem}\label{lem4}  \ \\
The Generalized Tempered Stable (GTS) distribution is self-decomposable, and its background driving L\'evy process $Y$ is the compound Poisson process with Poisson exponent.
  \begin{equation}
 \begin{aligned}
N(dx)=\left[\alpha_{+}\frac{\beta_{+} + \lambda_{+}x}{x^{1+\beta_{+}}}e^{-\lambda_{+}x} \boldsymbol{1}_{x>0} \right] dx  +  \left[\alpha_{-}\frac{\beta_{-} + \lambda_{-}|x|}{|x|^{1+\beta_{-}}} e^{-\lambda_{-}|x|}\boldsymbol{1}_{x<0}\right] dx \label{eq:l23bis}
  \end{aligned}
\end{equation} 
\end{theorem}
\begin{proof} \ \\
\noindent
Let  $\left[a_{\phi^{+}}, \sigma^{2}_{\phi^{+}}, M_{\phi^{+}}\right]$, the  L\'evy–Khintchine triplet of the Tempered stable distribution TS(\textbf{$\beta_{+}$}, \textbf{$\alpha_{+}$},\textbf{$\lambda_{+}$}) and $\left[b_{\varphi^{+}}, s^{2}_{\varphi^{+}}, N_{\varphi^{+}} \right]$, the  L\'evy–Khintchine triplet of the associated distribution.\\
The property (\ref{eq:l2851}) is applied by choosing $A=(y,+\infty)$
\begin{equation}
 \begin{aligned}
M_{\phi^{+}}((y,+\infty))=\int_{y}^{+\infty}\frac{\alpha_{+}}{x^{1+\beta_{+}}}e^{-\lambda_{+}x}dx=\int_{0}^{+\infty}N_{\varphi^{+}}(( e^{t}y,+\infty))dt =\int_{y}^{+\infty}\frac{N_{\varphi^{+}}((u,+\infty))}{u}du \label{eq:l2852}   
\end{aligned}
\end{equation} 
\noindent
In order to have the relation (\ref{eq:l2852}), we need to have:
\begin{equation*}
 \begin{aligned}
N_{\varphi^{+}}((u,+\infty))=\frac{\alpha_{+}}{u^{\beta_{+}}}e^{-\lambda_{+}u} 
\end{aligned}
\end{equation*} 
\noindent
The L\'evy density becomes
\begin{equation}
 \begin{aligned}
 N_{\varphi^{+}}(du)= - \frac{dN_{\varphi^{+}}((u,+\infty))}{du}du= \alpha_{+}(\lambda_{+} + \beta_{+}u^{-1})u^{-\beta_{+}}e^{-\lambda_{+}u} du
\end{aligned}
\end{equation} 
\noindent
The derivation of the relation (\ref{eq:l2852}) provides the relation between the L\'evy densities $M_{\phi^{+}}(dy)$ and $N_{\varphi^{+}}(dv)$ as follows.
\begin{equation}
 \begin{aligned}
M_{\phi^{+}}(dy)&=\int_{0}^{+\infty} N_{\varphi^{+}}( e^{t}dy)dt\\
&=\int_{0}^{+\infty}\alpha_{+}(\lambda_{+} + \beta_{+} e^{-t}y^{-1}) e^{(1-\beta_{+})t}y^{-\beta_{+}}e^{-\lambda_{+}y e^{t}}dt dy 
\end{aligned}
\end{equation} 
\noindent
We assume $\phi^{+}(\xi)$ is the characteristic function of the Tempered stable distribution TS(\textbf{$\beta_{+}$}, \textbf{$\alpha_{+}$},\textbf{$\lambda_{+}$}). 
   \begin{equation*}
 \begin{aligned}
\log(\phi^{+}(\xi))&=\int_{0}^{+\infty}(e^{iy\xi}-1)M_{\phi^{+}}(dy)\\
&=\int_{0}^{+\infty}\int_{0}^{+\infty}(e^{iy\xi}-1)\alpha_{+}(\lambda_{+} + \beta_{+} e^{-t}y^{-1}) e^{(1-\beta_{+})t}y^{-\beta_{+}}e^{-\lambda_{+}y e^{t}}dt dy \\
&=\int_{0}^{+\infty}\int_{0}^{+\infty}(e^{i\xi v e^{-t}}-1)\alpha_{+}(\lambda_{+} + \beta_{+}v^{-1})v^{-\beta_{+}}e^{-\lambda_{+}v}dt dv  \quad v=y e^{t}\\
&=\int_{0}^{1}\int_{0}^{+\infty}(e^{i\xi v u}-1)\alpha_{+}(\lambda_{+} + \beta_{+}v^{-1})v^{-\beta_{+}}e^{-\lambda_{+}v}dv\frac{du}{u}  \quad u= e^{t}\\
&=\int_{0}^{1}\log(\varphi^{+}(\xi u))\frac{du}{u}  =\int_{0}^{\xi}\log(\varphi^{+}(x))\frac{dx}{x} \quad \quad x=\xi u
\end{aligned}
\end{equation*} 
We have 
 \begin{align}
\Psi^{+}(\xi)=\log(\phi^{+}(\xi))=\int_{0}^{\xi}\log(\varphi^{+}(x))\frac{dx}{x} \label{eq:l08a}\\
  \log(\varphi^{+}(x))=\int_{0}^{+\infty}(e^{iv x}-1) N_{\varphi^{+}}(dv) \quad \quad N_{\varphi^{+}}(dx)&=\alpha_{+}\frac{\beta_{+} + \lambda_{+}x}{x^{1+\beta_{+}}} e^{-\lambda_{+}x} dx \label{eq:l08b}
 \end{align}
Following the same procedure on the left side, the Tempered stable distribution TS(\textbf{$\beta_{-}$}, \textbf{$\alpha_{-}$},\textbf{$\lambda_{-}$}), we have
 \begin{equation*}
 \begin{aligned}
\Psi^{-}(\xi)=\log(\phi^{-}(\xi))=\int_{0}^{\xi}\log(\varphi^{-}(x))\frac{dx}{x} \\
\log(\varphi^{-}(x))=\int_{0}^{+\infty}(e^{iv x}-1)N_{\varphi^{-}}(dv) \quad & \quad N_{\varphi^{-}}(dx)&=\alpha_{-}\frac{\beta_{-} + \lambda_{-}x}{x^{1+\beta_{-}}} e^{-\lambda_{-}x} dx\\
 \end{aligned}
\end{equation*} 
By putting the left and the right sides of the Tempered stable distribution together, we have
 \begin{equation}
 \begin{aligned}
 \log(\varphi^{+}(x)) +  \log(\varphi^{-}(-x))&= \int_{-\infty}^{+\infty}(e^{iv x}-1) N_{\varphi}(dx)\\
 \log(\phi^{+}(\xi)) +  \log(\phi^{-}(-\xi))&=\int_{0}^{\xi}( \log(\varphi^{+}(x)) +  \log(\varphi^{-}(-x))) \frac{dx}{x}\\
 &=\int_{0}^{\xi}\int_{-\infty}^{+\infty}(e^{iv x}-1) N_{\varphi}(dx) \frac{dx}{x} \label{eq:l271}
 \end{aligned}
\end{equation} 
Where
 \begin{equation*}
 \begin{aligned}
  N_{\varphi}(dx)=\left[\alpha_{+}\frac{\beta_{+} + \lambda_{+}x}{x^{1+\beta_{+}}} e^{-\lambda_{+}x} \boldsymbol{1}_{x>0} \right]dx + \left[\alpha_{-}\frac{\beta_{-} + \lambda_{-}|x|}{|x|^{1+\beta_{-}}}  e^{-\lambda_{-}|x|}\boldsymbol{1}_{x<0}\right] dx
 \end{aligned}
\end{equation*} 
\noindent
We have the characteristic exponent (\ref{eq:l271}) of  the Generalized Tempered Stable GTS(\textbf{$\mu$}, \textbf{$\beta_{+}$}, \textbf{$\beta_{-}$}, \textbf{$\alpha_{+}$},\textbf{$\alpha_{-}$}, \textbf{$\lambda_{+}$}, \textbf{$\lambda_{-}$})
  \begin{equation}
 \begin{aligned}
\Psi(\xi)=&Log\left(Ee^{i L\xi}\right)=\mu\xi i +   \log(\phi^{+}(\xi)) +  \log(\phi^{-}(-\xi))\\
&=\int_{0}^{\xi}\left( \mu x i + \int_{-\infty}^{+\infty}(e^{iv x}-1) N_{\varphi}(dx)\right)\frac{dx}{x}= \int_{0}^{\xi} \log(\varphi(x)\frac{dx}{x}\label{eq:l09a} 
\end{aligned}
\end{equation}
Where
  \begin{equation*}
 \begin{aligned}
\log(\varphi(x)) =\mu x i + \int_{-\infty}^{+\infty}(e^{iv x}-1) N_{\varphi}(dx)
\end{aligned}
\end{equation*}
 
\noindent 
We conclude from (\ref{eq:l09a}) that the Generalized Tempered Stable GTS is self-decomposable. 
\end{proof} 

\begin{theorem} \label{lem2} (Characteristic exponents of BDL process of the GTS distribution) \ \\
 \noindent
 The Generalized Tempered Stable (GTS) distribution is self-decomposable, and its background driving L\'evy process $Y$ has the following characteristic exponent.
   \begin{equation}
 \begin{aligned}
\log(\varphi(y))&=  \mu i y+ \alpha_{+}\Gamma(1 - \beta_{+})\frac{i y}{ (\lambda_{+} - i y)^{1- \beta_{+}}} + \alpha_{-}\Gamma(1 - \beta_{-})\frac{- i y}{ (\lambda_{-} + i y)^{1- \beta_{-}}}  \end{aligned}
\end{equation} 
\end{theorem} 

\begin{proof} \ \\
\noindent
From (\ref{eq:l08a}),  we have :%
 \begin{equation}
 \begin{aligned}
 \log(\varphi^{+}(x))&=\int_{0}^{+\infty}(e^{iv x}-1)\alpha_{+}(\lambda_{+} + \beta_{+}v^{-1})v^{-\beta_{+}}e^{-\lambda_{+}v}dv \\
 &=\sum_{k=1}^{\infty}\frac{(ix)^{k}}{k!}\alpha_{+}\left(\lambda_{+} \int_{0}^{+\infty}v^{k - \beta_{+}}e^{-\lambda_{+}v}dv + \beta_{+}\int_{0}^{+\infty}v^{-1-\beta_{+}}e^{-\lambda_{+}v}dv\right)\\
  &=\sum_{k=1}^{\infty}\frac{(ix)^{k}}{k!}\alpha_{+}\left(\lambda_{+}\frac{\Gamma(k+1-\beta_{+})}{\lambda_{+}^{k+1-\beta_{+}}}  + \beta_{+}\frac{\Gamma(k - \beta_{+})}{\lambda_{+}^{k - \beta_{+}}}\right)\\
 &=\alpha_{+}\lambda_{+}^{\beta_{+}}\Gamma(1-\beta_{+})\left(\sum_{k=1}^{\infty}\frac{\Gamma(k+1-\beta_{+})}{\Gamma(1-\beta_{+})k!}(\frac{ix}{\lambda_{+}})^{k}  - \sum_{k=1}^{\infty}\frac{\Gamma(k-\beta_{+})}{\Gamma(-\beta_{+})k!}(\frac{ix}{\lambda_{+}})^{k} \right)\\
&=\alpha_{+}\lambda_{+}^{\beta_{+}}\Gamma(1-\beta_{+})\left(\sum_{k=1}^{\infty}{\beta_{+} -1 \choose k}(-\frac{ix}{\lambda_{+}})^{k}  - \sum_{k=1}^{\infty}{\beta_{+} \choose k}(\frac{ix}{\lambda_{+}})^{k} \right)\\
&=\alpha_{+}\Gamma(1-\beta_{+})\left(\lambda_{+}(\lambda_{+} - ix)^{\beta_{+} -1}  - (\lambda_{+} - ix)^{\beta_{+}}  \right)\\
&=\alpha_{+}\Gamma(1-\beta_{+})\frac{ix}{(\lambda_{+} - ix)^{1 - \beta_{+}}}
  \end{aligned}
\end{equation}
Similarly, we have 
\begin{equation*}
 \begin{aligned}
\log(\varphi^{-}(x))&=\int_{0}^{+\infty}(e^{iv x}-1)\alpha_{-}(\lambda_{-} + \beta_{-}v^{-1})v^{-\beta_{-}}e^{-\lambda_{-}v}dv =\alpha_{-}\Gamma(1-\beta_{-})\frac{ix}{(\lambda_{-} - ix)^{1 - \beta_{-}}} 
\end{aligned}
\end{equation*} 
And
  \begin{equation*}
 \begin{aligned}
  \log(\varphi(x))&= \log(\varphi^{+}(x)) + \log(\varphi^{-}(-x))\\
 &= \alpha_{+}\Gamma(1-\beta_{+})\frac{ix}{(\lambda_{+} - ix)^{1 - \beta_{+}}} + \alpha_{-}\Gamma(1-\beta_{-})\frac{-ix}{(\lambda_{-} + ix)^{1 - \beta_{-}}}
\end{aligned}
\end{equation*} 
\end{proof} 


\subsection{ GTS distribution as  Background Driving L\'evy Process (BDLP)}

\begin{theorem}\label{lem6}  \ \\
The Generalized Tempered Stable (GTS) distribution is the Background Driving L\'evy Process (BDLP) of a self-decomposable random variable $X$. Then the self-decomposable random variable $X$ has the following  L\'evy–Khintchine triplet $\left[ \mu, 0, U \right]$.
\begin{equation}
\begin{aligned}
U(dx)=\left[\alpha_{+}\lambda_{+}^{\beta_{+}}\frac{\Gamma(-\beta_{+},\lambda_{+}x)}{x}\boldsymbol{1}_{x>0}\right]dx   + \left[\alpha_{-}\lambda_{-}^{\beta_{-}}\frac{\Gamma(-\beta_{-},\lambda_{-}x)}{x} \boldsymbol{1}_{x<0} \right]dx  \label {eq:lem64}
\end{aligned}
\end{equation}
 \end{theorem} 
 
\begin{proof} \ \\
\noindent
The L\'evy density of the Generalized Tempered Stable (GTS) distribution is defined in (\ref{eq:l23}) as follows.
 \begin{align}
W(dx) =\left(\frac{\alpha_{+}e^{-\lambda_{+}x}}{x^{1+\beta_{+}}} \boldsymbol{1}_{x>0}\right) dx  + \left(\frac{\alpha_{-}e^{-\lambda_{-}|x|}}{|x|^{1+\beta_{-}}} \boldsymbol{1}_{x<0}\right) dx =W^{+}(dx) + W^{-}(dx) \label{eq:lem61}  
 \end{align}
\noindent
Let  $\left[a^{+}, 0, W^{+}\right]$, the  L\'evy–Khintchine triplet of the Tempered stable distribution TS(\textbf{$\beta_{+}$}, \textbf{$\alpha_{+}$},\textbf{$\lambda_{+}$}) and $\left[b^{+}, 0, U^{+} \right]$, the  L\'evy–Khintchine triplet of the self-decomposable random variable $X$.\\

We apply the property (\ref{eq:l2851}) by choosing $A=(y,+\infty)$
\begin{equation}
 \begin{aligned}
U^{+}((y,+\infty))=\int_{y}^{+\infty}U^{+}(dx)=\int_{0}^{+\infty}W^{+}(( e^{t}y,+\infty))dt =\int_{y}^{+\infty}\frac{W^{+}((u,+\infty))}{u}du \label{eq:lem62}   
\end{aligned}
\end{equation} 
\noindent
In order to have the relation (\ref{eq:lem62}), we need to have:
\begin{equation*}
 \begin{aligned}
U^{+}(dx)=\frac{W^{+}((x,+\infty))}{x}dx
\end{aligned}
\end{equation*} 
The L\'evy density becomes
\begin{equation*}
 \begin{aligned}
U^{+}(dx)=\frac{W^{+}((x,+\infty))}{x}dx&=\frac{1}{x}\int_{x}^{+\infty}W^{+}(dy)dx=\frac{1}{x}\int_{x}^{+\infty}\frac{\alpha_{+}e^{-\lambda_{+}y}}{y^{1+\beta_{+}}}dydx  \\
&=\frac{\alpha_{+}\lambda_{+}^{\beta_{+}}}{x}\int_{\lambda_{+}x}^{+\infty}z^{-\beta_{+} - 1}e^{-z}dzdx \hspace{5mm}
   \hbox{$z=\lambda_{+}y$}\\
&=\alpha_{+}\lambda_{+}^{\beta_{+}}\frac{\Gamma(-\beta_{+},\lambda_{+}x)}{x} dx
\end{aligned}
\end{equation*} 

We have: 
\begin{equation}
 \begin{aligned}
U^{+}(dx)=\alpha_{+}\lambda_{+}^{\beta_{+}}\frac{\Gamma(-\beta_{+},\lambda_{+}x)}{x}dx \label {eq:lem63}
\end{aligned}
\end{equation} 
\noindent
We follow the same procedure on the left side of the Tempered stable distribution TS(\textbf{$\beta_{-}$}, \textbf{$\alpha_{-}$},\textbf{$\lambda_{-}$})  and we have:
\begin{equation}
 \begin{aligned}
U^{-}(dx)=\alpha_{-}\lambda_{-}^{\beta_{-}}\frac{\Gamma(-\beta_{-},\lambda_{-}x)}{x}dx \label {eq:lem63}
\end{aligned}
\end{equation} 

The L\'evy density of the random variable $X$  has the following expression
\begin{equation*}
 \begin{aligned}
U(dx)= U^{+}(dx) + U^{-}(dx)=\left[\alpha_{+}\lambda_{+}^{\beta_{+}}\frac{\Gamma(-\beta_{+},\lambda_{+}x)}{x}\boldsymbol{1}_{x>0} + \alpha_{-}\lambda_{-}^{\beta_{-}}\frac{\Gamma(-\beta_{-},\lambda_{-}|x|)}{|x|} \boldsymbol{1}_{x<0} \right]dx 
\end{aligned}
\end{equation*} 
The relation between the L\'evy densities $U^{+}(dx)$ and $W^{+}(dx)$ is provided if we derive the relation (\ref{eq:lem62}).
\begin{equation}
 \begin{aligned}
U^{+}(dy)&=\int_{0}^{+\infty} W^{+}( e^{t}dy)dt=\int_{0}^{+\infty}\alpha_{+}\frac{e^{-\lambda_{+}y e^{t}}}{(y e^{t})^{1+\beta_{+}}} e^{t}dt dy 
\end{aligned}
\end{equation} 
\noindent
We assume $\phi^{+}(\xi)$ is the characteristic function of the random variable $X^{+}$. 
   \begin{equation*}
 \begin{aligned}
\log(\phi^{+}(\xi))&=\int_{0}^{+\infty}(e^{iy\xi}-1)U^{+}(dy)=\int_{0}^{+\infty}\int_{0}^{+\infty}(e^{iy\xi}-1)\alpha_{+}\frac{e^{-\lambda_{+}y e^{t}}}{(y e^{t})^{1+\beta_{+}}} e^{t}dt dy \\
&=\int_{0}^{+\infty}\int_{0}^{+\infty}(e^{i\xi v e^{-t}}-1)\alpha_{+}\frac{e^{-\lambda_{+}v}}{v^{1+\beta_{+}}}dt dv  =\int_{0}^{1}\int_{0}^{+\infty}(e^{i\xi v u}-1)W^{+}(dv)dv\frac{du}{u} \\
&=\int_{0}^{1}\Psi^{+}(\xi u))\frac{du}{u}  =\int_{0}^{\xi}\Psi^{+}(x))\frac{dx}{x} \quad \quad x=\xi u
\end{aligned}
\end{equation*} 
Similarly, we have
   \begin{equation*}
 \begin{aligned}
\log(\phi^{-}(\xi)) =\int_{0}^{\xi}\Psi^{-}(x))\frac{dx}{x} 
\end{aligned}
\end{equation*} 
We assume $\phi(\xi)$ is the characteristic function of the random variable $X$. 
 \begin{equation}
 \begin{aligned}
 \log(\phi(\xi))&= \mu x i + \log(\phi^{+}(\xi)) + \log(\phi^{-}(-\xi))\\
 &= \int_{0}^{\xi}(\mu x i + \Psi^{+}(x)+ \Psi^{-}(-x)) \frac{dx}{x}= \int_{0}^{\xi} \Psi(x) \frac{dx}{x}
\end{aligned}
\end{equation} 
The random variable $X$ is self-decomposable, and its background driving L\'evy process is the Generalized Tempered Stable (GTS) distribution.
\end{proof} 

\subsubsection{ Asymptotic behavior of $U(dx)$}

As $x \to 0$, we have
   \begin{equation*}
 \begin{aligned}
&\Gamma(-\beta_{+},\lambda_{+}x) \sim \frac{1}{\beta_{+}\lambda_{+}^{\beta_{+}}}\frac{1}{x^{\beta_{+}}}\\
&U^{+}(dx) \sim \frac{\alpha_{+}}{\beta_{+}}\frac{1}{x^{1+\beta_{+}}}dx
  \end{aligned}
\end{equation*}  
The L\'evy density of the self-decomposable random variable $X$ becomes the L\'evy density of a stable distribution  \cite{Giacomo}.
\begin{equation}
 \begin{aligned}
U(dx)\sim\left[\frac{\alpha_{+}}{\beta_{+}}\frac{1}{x^{1+\beta_{+}}}\boldsymbol{1}_{x>0}\right]dx+ \left[\frac{\alpha_{-}}{\beta_{-}}\frac{1}{x^{1+\beta_{-}}} \boldsymbol{1}_{x<0} \right]dx \label {eq:lem65}\end{aligned}
\end{equation}
As $x \to +\infty$, we have

   \begin{equation*}
 \begin{aligned}
&\Gamma(-\beta_{+},\lambda_{+}x) \sim \frac{1}{\lambda_{+}^{ 1 + \beta_{+}}}\frac{e^{-\lambda_{+}x}}{x^{1+\beta_{+}}}\\
&U^{+}(dx) \sim \frac{\alpha_{+}}{\lambda_{+}} \frac{e^{-\lambda_{+}x}}{x^{2+\beta_{+}}}dx
  \end{aligned}
\end{equation*}  
The L\'evy density of the self-decomposable random variable $X$ becomes the Generalized Tempered Stable (GTS) distribution.
\begin{equation}
 \begin{aligned}
U(dx)\sim\left[\frac{\alpha_{+}}{\lambda_{+}} \frac{e^{-\lambda_{+}x}}{x^{2+\beta_{+}}}\boldsymbol{1}_{x>0}\right]dx + \left[\frac{\alpha_{-}}{\lambda_{-}}\frac{e^{-\lambda_{-}x}}{x^{2+\beta_{-}}} \boldsymbol{1}_{x<0} \right]dx \label {eq:lem66}\end{aligned}
\end{equation}
U(dx) behaves like a $\alpha$-stable distribution near zero, with a GTS distribution at the tails.\\

\begin{theorem} \label{lem7} (Characteristic Exponents) \ \\
 \noindent
The Generalized Tempered Stable (GTS) distribution is the Background Driving L\'evy Process (BDLP) of a self-decomposable random variable $X$. Then the characteristic exponents of the self-decomposable random variable $X$ is:
\begin{align}
\gamma(\xi)=\mu\xi i + \gamma^{+}(\xi) + \gamma^{-}( - \xi)   \label {eq:l2720}
 \end{align}
 Where 
\begin{equation*}
\begin{aligned}
\gamma^{+}(\xi)&=\alpha_{+}\Gamma(-\beta_{+}) \int_{0}^{1}\frac{(\lambda_{+} - i\xi u)^{\beta_{+}} - \lambda_{+}^{\beta_{+}}}{u} du\\
\gamma^{-}(\xi)&=\alpha_{-}\Gamma(-\beta_{-})\int_{0}^{1}\frac{(\lambda_{-} - i\xi u)^{\beta_{-}} - \lambda_{-}^{\beta_{-}}}{u}du \\
 \end{aligned}
\end{equation*} 

\end{theorem} 

\begin{proof} \ \\
\noindent
\begin{equation}
\begin{aligned}
\Gamma(-\beta_{+},\lambda_{+}y) &=\int_{\lambda_{+}y}^{+\infty}x^{-\beta_{+}-1}e^{-x}dx=\int_{1}^{+\infty}(\lambda_{+}y)^{-\beta_{+}}u^{-\beta_{+}-1}e^{-\lambda_{+}yu}du
 \label{eq:l081a}
 \end{aligned}
\end{equation}

\begin{equation}
\begin{aligned}
\int_{0}^{+\infty}(e^{iy\xi}-1)\frac{ \Gamma(-\beta_{+},\lambda_{+}y)}{y}dy&=\sum_{k=1}^{\infty} \frac{(i\xi)^{k}}{k!}\left[\int_{0}^{+\infty}\frac{ y^{k} \Gamma(-\beta_{+},\lambda_{+}y)}{y}dy\right] \hspace{5mm}  \hbox{recall (\ref{eq:l081a})}\\
&=\sum_{k=1}^{\infty} \frac{(i\xi)^{k}}{k!}\left[\int_{0}^{+\infty}y^{k}\int_{1}^{+\infty}\lambda_{+}^{-\beta_{+}}(uy)^{-\beta_{+}-1}e^{-\lambda_{+}yu}dudy\right]\\
&=\sum_{k=1}^{\infty} \frac{(i\xi)^{k}}{k!}\left[\int_{1}^{+\infty}\lambda_{+}^{-\beta_{+}}u^{-\beta_{+}-1}\int_{0}^{+\infty}y^{k-\beta_{+}-1}e^{-\lambda_{+}yu}dydu\right]\\
&=\sum_{k=1}^{\infty} \frac{(i\xi)^{k}}{k!}\left[\frac{\Gamma(k-\beta_{+})}{\lambda_{+}^{k}}\int_{1}^{+\infty}u^{-k-1}du\right]\\
&=\Gamma(-\beta_{+})\sum_{k=1}^{\infty}\frac{\Gamma(k-\beta_{+})}{k!\Gamma(-\beta_{+})k} (\frac{i\xi}{\lambda_{+}})^{k} \label{eq:l086}
\end{aligned}
\end{equation}

\begin{equation*}
\begin{aligned}
\frac{\Gamma(j-\beta)}{\Gamma(-\beta)\Gamma(j+1)}=(-1)^{j}{\beta \choose j}    \hspace{5mm}
   \hbox{ratio of falling factorials (Pochhammer symbol)} 
\end{aligned}
\end{equation*}

\begin{equation}
\begin{aligned}
\xi\frac{d\left(\sum_{k=1}^{\infty}\frac{\Gamma(k - \beta_{+})}{k!\Gamma(-\beta_{+})} \frac{1}{k} (\frac{i\xi}{\lambda_{+}})^{k}\right)}{d\xi}= \sum_{k=1}^{\infty}\frac{\Gamma(k - \beta_{+})}{k!\Gamma(-\beta_{+})} (\frac{i\xi}{\lambda_{+}})^{k}= \sum_{k=1}^{\infty}{{\beta_{+}}  \choose k} (\frac{-i\xi}{\lambda_{+}})^{k} \label{eq:l084}
\end{aligned}
\end{equation}

Newton's generalized binomial theorem can be applied, and (\ref{eq:l084}) becomes
\begin{equation}
\begin{aligned}
\xi\frac{d\left(\sum_{k=1}^{\infty}\frac{\Gamma(k+1-\beta_{+})}{k!\Gamma(1-\beta_{+})} \frac{1}{k} (\frac{i\xi}{\lambda_{+}})^{k}\right)}{d\xi} = (1-\frac{i\xi}{\lambda_{+}})^{\beta_{+}} -1\hspace{5mm}
   \forall \xi \in \mathbb{R}  \label{eq:l085}
\end{aligned}
\end{equation}
The first-order differential equation can be solved as follows
\begin{equation*}
\begin{aligned}
\sum_{k=1}^{\infty}\frac{\Gamma(k+1-\beta_{+})}{k!\Gamma(1-\beta_{+})} \frac{1}{k} (\frac{i\xi}{\lambda_{+}})^{k} &=\frac{1}{\lambda_{+}^{\beta_{+}}}\int_{0}^{\xi}\frac{(\lambda_{+} - i y)^{\beta_{+}} - \lambda_{+}^{\beta_{+}}}{y} dy\\
&=\frac{1}{\lambda_{+}^{\beta_{+}}}\int_{0}^{1}\frac{(\lambda_{+} - i\xi u)^{\beta_{+}} - \lambda_{+}^{\beta_{+}}}{u} du \quad \quad u=\frac{y}{\xi}
\end{aligned}
\end{equation*}
From the previous development (\ref{eq:l086}), we have :
\begin{equation}
\begin{aligned}
\int_{0}^{+\infty}(e^{iy\xi}-1)U^{+}(dx) = \alpha_{+}\Gamma(-\beta_{+}) \int_{0}^{1}\frac{(\lambda_{+} - i\xi u)^{\beta_{+}} - \lambda_{+}^{\beta_{+}}}{u} du \label{eq:l086a}
\end{aligned}
\end{equation}
\noindent
We have the same results on $U^{-}(dx)$
\begin{equation}
\begin{aligned}
\int_{0}^{+\infty}(e^{iy\xi}-1)U^{-}(dx) =\alpha_{-}\Gamma(-\beta_{-})\int_{0}^{1}\frac{(\lambda_{-} - i\xi u)^{\beta_{-}} - \lambda_{-}^{\beta_{-}}}{u}du \label{eq:l086b}
\end{aligned}
\end{equation}
\end{proof} 

\section{Fitting General Tempered Stable Distribution Results}\label{sec5}

\subsection{Review of the Maximum Likelihood Method} 
\noindent
From a probability density function $f(x, V)$ with parameter $V$ of size $p=7$ and the sample data $X$ of size $m$,  we definite  the likelihood function and its derivatives as follows: 
\begin{align}
 L(x,V) &= \prod_{j=1}^{m} f(x_{j},V) \quad &
 l(x,V) &= \sum_{j=1}^{m} log(f(x_{j},V))  \label {eq:l32}
  \end{align}
  \begin{equation}
 \begin{aligned}
 \frac{dl(x,V)}{dV_j} &= \sum_{i=1}^{m} \frac{\frac{df(x_{i},V)}{dV_j}}{f(x_{i},V)} \quad \quad 
 \frac{d^{2}l(x,V)}{dV_{k}dV_{j}} &= \sum_{i=1}^{m} \left(\frac{\frac{d^{2}f(x_{i},V)}{dV_{k}dV_{j}}}{f(x_{i},V)}- \frac{\frac{df(x_{i},V)}{dV_{k}}}{f(x_{i},V)}\frac{\frac{df(x_{i},V)}{dV_j}}{f(x_{i},V)}\right). \label {eq:l35}
\end{aligned}
\end{equation}
\noindent
To perform the maximum of the likelihood function (\ref{eq:l32}), the quantities  $\frac{df(x, V)}{dV_j}$ and $\frac{d^{2}f(x, V)}{dV_{k}dV_{j}}$ in (\ref{eq:l35}) are the first and second order derivatives of the probability density and should be computed with high accuracy.
\noindent
The quantities $\frac{d^{2}l(x, V)}{dV_{k}dV_{j}}$ are critical in computing the Hessian matrix and the Fisher information matrix.\\

\noindent
Given the parameter $V=(\textbf{$\mu$}, \textbf{$\beta_{+}$}, \textbf{$\beta_{-}$}, \textbf{$\alpha_{+}$},\textbf{$\alpha_{-}$}, \textbf{$\lambda_{+}$}, \textbf{$\lambda_{-}$})$ and the sample data set $X$, we derive from (\ref{eq:l35}) the following vector and matrix: (\ref{eq:l36}) .
  \begin{align}
 I'(X,V) =\left(\frac{dl(x,V)}{dV_j}\right)_{1 \leq j \leq p}  \quad  \quad  I''(X,V) = \left(\frac{d^{2}l(x,V)}{dV_{k}dV_{j}}\right)_{\substack{{1 \leq k \leq p} \\ {1 \leq j \leq p}}}. \label {eq:l36}
 \end{align}
\noindent
The advanced fast FRFT scheme developed previously is used to compute the likelihood function (\ref{eq:l32}) and its derivatives (\ref{eq:l36}) in the optimization process. \\

\noindent
The local solution $V^{0}$ should meet the following requirements:
  \begin{align}
 I'(x,V^{0})=0 \quad \quad U^{T}\mathbf{I''(X,V^{0})}U \leq 0\hspace{5mm}  \hbox{ $\forall U \in \mathbb{R}^{p}$}.\label{eq:l37}
 \end{align}
The inequality in (\ref{eq:l37}) is met when $I''(x, V^{0})$ is a negative semi-definite matrix. The Newton-Raphson iteration algorithm provides the numerical solution (\ref{eq:l38}):
  \begin{align}
V^{n+1}=V^{n}-{\left(I''(x,V^{n})\right)^{-1}}I'(x,V^{n}). \label{eq:l38}
 \end{align}
\noindent
See \cite{giudici2013wiley} for details on maximum likelihood and Newton-Raphson iteration procedure.

\subsection{GTS Parameter Estimation from S$\&$P 500 Index}
\noindent
The results of the GTS parameter estimation from S$\&$P 500 return data are reported in Table \ref{tab1}. As expected, the stability indexes (\textbf{$\beta_{-}$},\textbf{$\beta_{+}$}), the process intensities (\textbf{$\alpha_{-}$},\textbf{$\alpha_{+}$}), and the decay rate (\textbf{$\lambda_{-}$},\textbf{$\lambda_{+}$}) are all positive. The results show that $0\le \beta_{+} \le 1$ and $0\le \beta_{-} \le 1$, S$\&$P 500 return is of infinite activity process (\ref{eq:l24}) with infinite number of jumps in any given time interval; S$\&$P 500 is also a finite variance process (\ref{eq:l25}).
\begin{table}[ht]
\vspace{-0.3cm}
\caption{ FRFT Maximum Likelihood GTS Parameter Estimation}
\label{tab1}
\vspace{-0.3cm}
\centering
\begin{tabular}{@{}lccccccc@{}}
\toprule
\textbf{Model} & \textbf{$\mu$} & \textbf{$\beta_{+}$} & \textbf{$\beta_{-}$} & \textbf{$\alpha_{+}$} & \textbf{$\alpha_{-}$}  & \textbf{$\lambda_{+}$}  & \textbf{$\lambda_{-}$}  \\ \midrule
\textbf{GTS} & -0.693477 & 0.682290 & 0.242579 & 0.458582 & 0.414443 & 0.822222 & 0.727607  \\ \bottomrule
\end{tabular}
\vspace{-0.2cm}
\end{table}

\noindent
As shown in Table \ref{tab1}, we have \textbf{$\alpha_{-}$} $\leq$ \textbf{$\alpha_{+}$} and the arrival rate of the positive jump is more than that of the negative jump. In addition, when \textbf{$\lambda_{-}$} $\leq$ \textbf{$\lambda_{+}$}, the S$\&$P 500 return is a bit left-skewed distribution. Based on $\alpha$ and $\lambda$, the tail distribution is thicker on the negative side of the S$\&$P 500 return distribution ($X_{-}$) than on the positive side of the distribution.\\

\noindent
The Newton-Raphson iteration algorithm (\ref{eq:l38}) was implemented, and the results are reported in Table \ref{tab2}. Each raw has eleven columns made of the iteration number, the seven parameters \textbf{$\mu$}, \textbf{$\beta_{+}$}, \textbf{$\beta_{-}$}, \textbf{$\alpha_{+}$}, \textbf{$\alpha_{-}$}, \textbf{$\lambda_{+}$}, \textbf{$\lambda_{-}$} and three statistical indicators: the log-likelihood (\textbf{$Log(ML)$}), the norm of the partial derivatives (\textbf{$||\frac{dLog(ML)}{dV}||$}), the maximum value of the eigenvalues (\textbf{$Max Eigen Value$}). The statistical indicators aim at checking if the two necessary and sufficient conditions described in equations(\ref{eq:l37}) are all met. \textbf{$Log(ML)$} displays the value of the Naperian logarithm of the likelihood function $L(x, V)$ as described in (\ref{eq:l32}); \textbf{$||\frac{dLog(ML)}{dV}||$} displays the value of the norm of the first derivatives (\textbf{$\frac{dl(x, V)}{dV_j}$}) described in equations (\ref{eq:l35}); and \textbf{$Max Eigen Value$} displays the maximum value of the seven eigenvalues generated by the Hessian Matrix (\textbf{$\frac{d^{2}l(x, V)}{dV_{k}dV_{j}}$}) as described in equations (\ref{eq:l36}).
\begin{table*}[ht]
\vspace{-0.3cm}
\centering
\caption{GTS Parameter Estimations from S$\&$P 500 Index}
\label{tab2}
\resizebox{13cm}{!}{%
\begin{tabular}{ccccccccccc}
\hline
\textbf{$Iterations$} & \textbf{$\mu$} & \textbf{$\beta_{+}$} & \textbf{$\beta_{-}$} & \textbf{$\alpha_{+}$} & \textbf{$\alpha_{-}$} & \textbf{$\lambda_{+}$} & \textbf{$\lambda_{-}$} & \textbf{$Log(ML)$} & \textbf{$||\frac{dLog(ML)}{dV}||$} & \textbf{$Max Eigen Value$} \\ \hline
1  & -0.5266543 & 0.67666185 & 0.43662658 & 0.43109407 & 0.32587754 & 0.85012595 & 0.60145985 & -4664.7647 & 289.206866 & 48.3287566 \\
2  & -0.5459715 & 0.67059496 & 0.42415971 & 0.44989796 & 0.34810329 & 0.80921924 & 0.60289955 & -4660.2156 & 35.9853332 & -6.0431046 \\
3  & -0.7108484 & 0.66676659 & 0.20613817 & 0.46963149 & 0.41246296 & 0.8368645  & 0.73679498 & -4663.5777 & 1082.00304 & 449.252496 \\
4  & -0.6704327 & 0.66891111 & 0.11250993 & 0.46528242 & 0.50028087 & 0.83422813 & 0.86330646 & -4660.5268 & 135.684998 & 15.1103621 \\
5  & -0.739816  & 0.66779246 & 0.09097181 & 0.48302677 & 0.45922815 & 0.85551132 & 0.81320574 & -4660.0208 & 45.7082317 & 10.8419307 \\
6  & -0.6517205 & 0.65580459 & 0.19673172 & 0.47996977 & 0.42743878 & 0.85248358 & 0.75350741 & -4659.833  & 46.3329967 & 11.8534517 \\
7  & -0.8136505 & 0.71997129 & 0.21558026 & 0.44023152 & 0.41952419 & 0.79423596 & 0.74025892 & -4662.4815 & 1187.98211 & 166.006596 \\
8  & -0.7805857 & 0.70635344 & 0.22952102 & 0.44666893 & 0.4166063  & 0.80360771 & 0.7334486  & -4659.7758 & 85.6584355 & -3.4305336 \\
9  & -0.7543747 & 0.69912144 & 0.23468747 & 0.45029318 & 0.41542774 & 0.80935861 & 0.73077041 & -4659.1944 & 1.07446211 & -0.7993855 \\
10 & -0.7533784 & 0.69885561 & 0.23480071 & 0.45042677 & 0.41541378 & 0.80956671 & 0.73072468 & -4659.1943 & 1.03706137 & -0.8136858 \\
11 & -0.752414  & 0.69859801 & 0.2349108  & 0.45055614 & 0.41540021 & 0.8097682  & 0.73068024 & -4659.1942 & 1.00184387 & -0.8273569 \\
12 & -0.7514794 & 0.69834808 & 0.23501796 & 0.45068156 & 0.41538701 & 0.80996353 & 0.73063702 & -4659.1942 & 0.96860649 & -0.8404525 \\
13 & -0.7496917 & 0.69786927 & 0.23522417 & 0.4509216  & 0.41536164 & 0.81033731 & 0.73055392 & -4659.194  & 0.90739267 & -0.8650984 \\
14 & -0.7471903 & 0.69719768 & 0.23551545 & 0.45125776 & 0.41532584 & 0.81086063 & 0.73043671 & -4659.1938 & 0.82669598 & -0.8987475 \\
15 & -0.7463993 & 0.69698491 & 0.23560822 & 0.45136413 & 0.41531445 & 0.8110262  & 0.73039942 & -4659.1937 & 0.80232744 & -0.9091961 \\
16 & -0.7456279 & 0.69677723 & 0.23569899 & 0.45146791 & 0.41530331 & 0.8111877  & 0.73036294 & -4659.1937 & 0.77907591 & -0.9193009 \\
17 & -0.7434222 & 0.69618233 & 0.23596025 & 0.45176484 & 0.41527126 & 0.81164974 & 0.73025805 & -4659.1935 & 0.71530396 & -0.9477492 \\
18 & -0.7427203 & 0.69599272 & 0.2360439  & 0.45185938 & 0.41526101 & 0.81179683 & 0.73022449 & -4659.1934 & 0.69583386 & -0.9566667 \\
19 & -0.7376152 & 0.69460887 & 0.23665992 & 0.45254797 & 0.41518555 & 0.81286777 & 0.72997772 & -4659.193  & 0.56549166 & -1.0197019 \\
20 & -0.7353516 & 0.69399264 & 0.23693733 & 0.45285383 & 0.41515159 & 0.81334323 & 0.72986678 & -4659.1929 & 0.51375646 & -1.0466686 \\
21 & -0.7342715 & 0.69369811 & 0.23707047 & 0.45299987 & 0.41513529 & 0.81357023 & 0.72981356 & -4659.1928 & 0.47702329 & -1.0633442 \\
22 & -0.727748  & 0.69191195 & 0.23788938 & 0.45388257 & 0.41503509 & 0.81494168 & 0.72948683 & -4659.1924 & 0.30254549 & -1.1564316 \\
23 & -0.7269375 & 0.69168915 & 0.23799284 & 0.45399234 & 0.41502243 & 0.81511217 & 0.72944561 & -4659.1924 & 0.28659447 & -1.1671214 \\
24 & -0.7261524 & 0.69147312 & 0.23809343 & 0.45409871 & 0.41501011 & 0.81527736 & 0.72940554 & -4659.1923 & 0.27235218 & -1.177299  \\
25 & -0.7074205 & 0.68625723 & 0.24059125 & 0.45665159 & 0.41470173 & 0.81923664 & 0.72841073 & -4659.1916 & 0.15127988 & -1.3705213 \\
26 & -0.7028887 & 0.68497607 & 0.24122265 & 0.45727561 & 0.4146222  & 0.82020261 & 0.72815815 & -4659.1915 & 0.12096032 & -1.4033087 \\
27 & -0.6987925 & 0.68381109 & 0.24180316 & 0.45784257 & 0.4145479  & 0.82107955 & 0.72792474 & -4659.1914 & 0.07862954 & -1.4283679 \\
28 & -0.6934505 & 0.68228736 & 0.24256739 & 0.45858391 & 0.41444901 & 0.82222574 & 0.72761631 & -4659.1914 & 0.75324649 & -1.6442249 \\
29 & -0.6934765 & 0.68229002 & 0.24257963 & 0.45858236 & 0.41444404 & 0.82222285 & 0.72760762 & -4659.1914 & 9.93E-06   & -1.4542674 \\
30 & -0.6934774 & 0.68229032 & 0.24257975 & 0.45858219 & 0.41444396 & 0.8222226  & 0.7276075  & -4659.1914 & 7.4193E-08 & -1.4542632\\ \hline
\end{tabular}%
}
\vspace{-0.3cm}
\end{table*}

\noindent
 As shown in Table \ref{tab2}, the log-likelihood (\textbf{$Log(ML)$}) value starts at $-4664.7647$ and increases to a limit of $-4659.1914$; the \textbf{$||\frac{dLog(ML)}{dV}||$} value starts at $289.206866$ and decreases to $0$; and the maximum value of the eigenvalues (\textbf{$Max Eigen Value$}) starts at $48.3287566$ and converges to $-1.4542632$, which is negative. The Hessian matrix in (\ref{eq:l36}) is a negative semi-definite matrix at a converged solution. Therefore, both conditions in (\ref{eq:l37}) are met, and we have a locally optimal solution.
\begin{figure}[ht]
\vspace{-0.3cm}
    \centering
  \begin{subfigure}[b]{0.4\linewidth}
    \includegraphics[width=\linewidth]{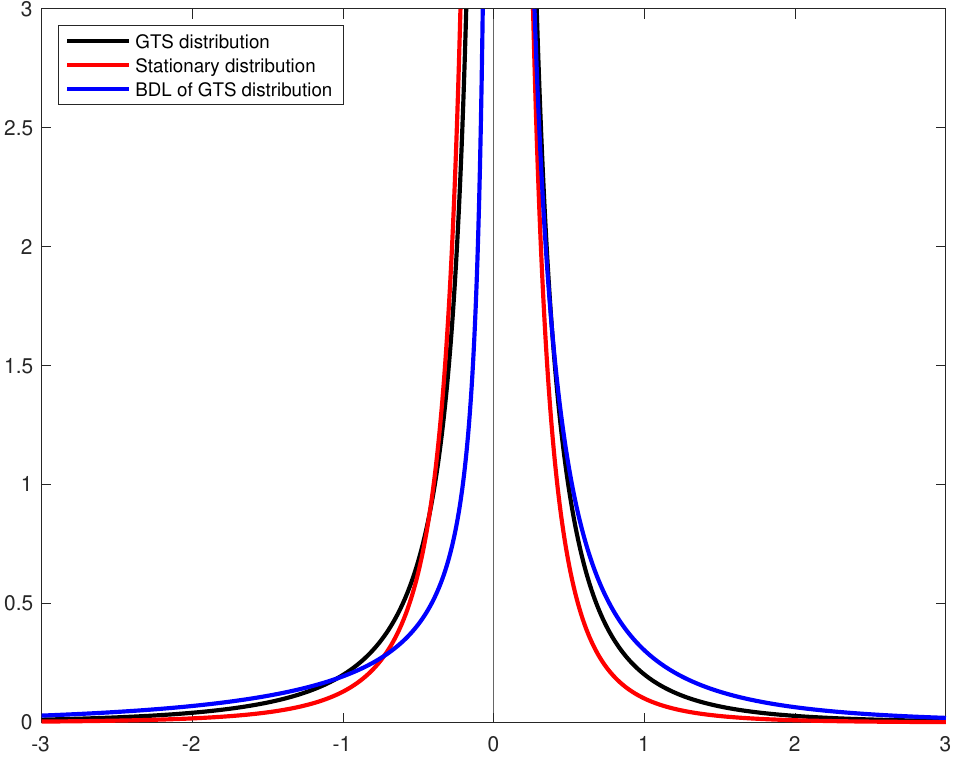}
\vspace{-0.4cm}
     \caption{L\'evy densities }
         \label{fig911}
  \end{subfigure}
  \begin{subfigure}[b]{0.5\linewidth}
    \includegraphics[width=\linewidth]{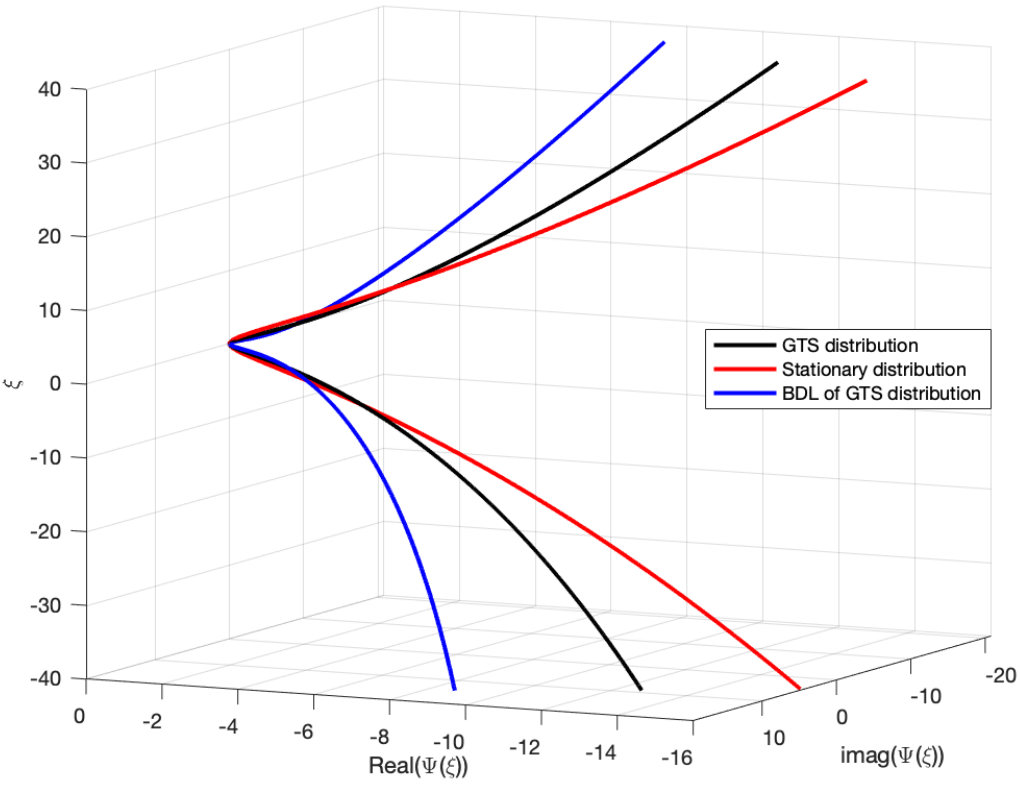}
\vspace{-0.6cm}
     \caption{Characteric exponents}
         \label{fig921}
          \end{subfigure}
\vspace{-0.6cm}
  \caption{ S\&P500 index Data: $\mu=-0.693477$, $\beta_{+}=0.682290$, $\beta_{-}=0.242579$, $\alpha_{+}=0.458582$, $\alpha_{-}=0.414443$, $\lambda_{+}=0.822222$, $\lambda_{-}=0.727607$}
  \label{fig01}
\vspace{-0.6cm}
\end{figure}
\newpage
\subsection{GTS Parameter Estimation: Bitcoin}
\noindent
The results of the GTS parameter estimation from Bitcoin returns are reported in Table \ref{tab3}. As expected, the stability indexes (\textbf{$\beta_{-}$},\textbf{$\beta_{+}$}), the process intensities (\textbf{$\alpha_{-}$},\textbf{$\alpha_{+}$}), and the decay rate (\textbf{$\lambda_{-}$},\textbf{$\lambda_{+}$}) are all positive. The results show that $0\le \beta_{+} \le 1$ and $0\le \beta_{-} \le 1$, Bitcoin return is of infinite activity process (\ref{eq:l24}) with infinite number of jumps in any given time interval; Bitcoin is also a finite variance process (\ref{eq:l25}).
\begin{table}[ht]
\caption{ FRFT Maximum Likelihood GTS Parameter Estimation}
\label{tab3}
\vspace{-0.3cm}
\centering
\begin{tabular}{@{}lccccccc@{}}
\toprule
\textbf{Model} & \textbf{$\mu$} & \textbf{$\beta_{+}$} & \textbf{$\beta_{-}$} & \textbf{$\alpha_{+}$} & \textbf{$\alpha_{-}$}  & \textbf{$\lambda_{+}$}  & \textbf{$\lambda_{-}$}  \\ \midrule
\textbf{GTS}  & -0.736924  & 0.461378 & 0.267178 & 0.810017 & 0.517347 & 0.215628 & 0.191937   \\ \bottomrule
\end{tabular}
\vspace{-0.3cm}
\end{table}

\noindent
 As shown in Table \ref{tab3}, we have \textbf{$\alpha_{-}$} $\leq$ \textbf{$\alpha_{+}$} and the arrival rate of the positive jump is significantly more than that of the negative jump. In addition, when we have \textbf{$\lambda_{-}$} $\leq$ \textbf{$\lambda_{+}$} then Bitcoin return is a left-skewed distribution. According to the values of  the process intensity ($\alpha$) and the decay rate($\lambda$), Bitcoin tail distribution is thicker on the negative side .\\
 
\noindent
The Newton-Raphson iteration algorithm (\ref{eq:l38}) was implemented, and the results of the iteration process are reported in Table \ref{tab4}. Table \ref{tab4} and Table \ref{tab2} have similar structures. As shown in Table \ref{tab4}, the log-likelihood (\textbf{$Log(ML)$}) value starts at $-10036.656$ and increases to a limit of $-9751.2193$; the \textbf{$||\frac{dLog(ML)}{dV}||$} value starts at $3188.1469$ and decreases to almost $0$; and the maximum value of the eigenvalues (\textbf{$Max Eigen Value$}) starts at $2854.2231$ and converges to $-0.000436$, which is negative. The Hessian matrix in (\ref{eq:l36}) is a negative semi-definite matrix at a converged solution. Therefore, both conditions in (\ref{eq:l37}) are met, and we have a locally optimal solution.
\begin{table*}[ht]
\vspace{-0.6cm}
\centering
\caption{GTS Parameter Estimation from Bitcoin Returns}
\label{tab4}
\resizebox{13cm}{!}{%
\begin{tabular}{ccccccccccc}
\hline
\textbf{$Iterations$} & \textbf{$\mu$} & \textbf{$\beta_{+}$} & \textbf{$\beta_{-}$} & \textbf{$\alpha_{+}$} & \textbf{$\alpha_{-}$} & \textbf{$\lambda_{+}$} & \textbf{$\lambda_{-}$} & \textbf{$Log(ML)$} & \textbf{$||\frac{dLog(ML)}{dV}||$} & \textbf{$Max Eigen Value$} \\ \hline
1 & -0.5973 & 0.3837 & 0.4206 & 1.872 & 0.694 & 0.5333 & 0.219 & -10036.656 & 3188.1469 & 2854.2231 \\
2 & -1.2052311 & 0.35850753 & 0.44451399 & 2.00512495 & 0.62313779 & 0.53840744 & 0.20635235 & -9912.1541 & 1971.95236 & 1518.15345 \\
3 & -1.624469 & 0.31755508 & 0.45798507 & 2.23748588 & 0.58475918 & 0.56388213 & 0.17998332 & -9865.5242 & 952.027546 & 716.347539 \\
4 & -1.9959515 & 0.2615866 & 0.47359748 & 2.57864607 & 0.5494987 & 0.60990478 & 0.16377681 & -9862.2551 & 440.489423 & 295.233361 \\
5 & -2.8340545 & -0.0033236 & 0.55268657 & 4.51611949 & 0.47359381 & 0.87264266 & 0.130233 & -9911.337 & 510.396179 & 60.6518499 \\
6 & -2.701598 & 0.07378625 & 0.54127007 & 3.76512451 & 0.47819392 & 0.7743675 & 0.13294829 & -9896.0802 & 611.613266 & 117.692562 \\
7 & -1.4130144 & 0.37715064 & 0.48510637 & 1.50024644 & 0.5523787 & 0.39998447 & 0.15474644 & -9804.2548 & 801.962681 & 307.039536 \\
8 & -1.049984 & 0.45100489 & 0.4741659 & 1.13318569 & 0.56070768 & 0.30817133 & 0.15667085 & -9775.871 & 728.753228 & 295.768229 \\
9 & -0.8211241 & 0.50542632 & 0.46591368 & 0.90824445 & 0.54487666 & 0.23555274 & 0.1533588 & -9756.7912 & 397.102674 & 133.652433 \\
10 & -0.7319208 & 0.53998085 & 0.44372538 & 0.81377008 & 0.54717214 & 0.19930192 & 0.15902826 & -9752.0888 & 89.8942291 & 12.4808028 \\
11 & -1.2317854 & 0.58475806 & 0.23780021 & 0.81260395 & 0.49719481 & 0.18597445 & 0.19343462 & -9751.3131 & 32.1533364 & -2.5188051 \\
12 & -0.7150133 & 0.4721273 & 0.29761245 & 0.80521805 & 0.52356611 & 0.21204916 & 0.18644449 & -9750.6869 & 5.52013773 & 0.11490414 \\
13 & -0.7271098 & 0.45992936 & 0.27054512 & 0.80932883 & 0.51829775 & 0.21585683 & 0.19139837 & -9751.2179 & 1.13726334 & 0.19172007 \\
14 & -0.7750442 & 0.46778988 & 0.25646086 & 0.81245021 & 0.51411912 & 0.21450942 & 0.19360685 & -9751.2252 & 1.00910053 & -0.9423667 \\
15 & -0.7090246 & 0.45678856 & 0.27535855 & 0.80819689 & 0.51977265 & 0.21641582 & 0.19065445 & -9751.2151 & 0.68576109 & 0.57268124 \\
16 & -0.7494266 & 0.46338031 & 0.26341052 & 0.81084098 & 0.51624184 & 0.21529106 & 0.19253194 & -9751.2211 & 0.65747338 & -0.2888343 \\
17 & -0.7429071 & 0.46233754 & 0.26537785 & 0.8104111 & 0.51681851 & 0.21546689 & 0.19222166 & -9751.2201 & 0.65058847 & -0.1361475 \\
18 & -0.7401206 & 0.46189156 & 0.2662183 & 0.81022749 & 0.51706495 & 0.21554213 & 0.19208913 & -9751.2197 & 0.64914609 & -0.0724136 \\
19 & -0.7369888 & 0.4613886 & 0.26715947 & 0.81002148 & 0.51734136 & 0.21562718 & 0.19194083 & -9751.2193 & 0.64887078 & -0.0018685 \\
20 & -0.7369633 & 0.4613845 & 0.26716714 & 0.8100198 & 0.51734361 & 0.21562787 & 0.19193962 & -9751.2193 & 0.64887314 & -0.0012986 \\
21 & -0.7369455 & 0.46138165 & 0.26717247 & 0.81001863 & 0.51734517 & 0.21562835 & 0.19193878 & -9751.2193 & 0.64887482 & -0.0009025 \\
22 & -0.7369332 & 0.46137967 & 0.26717617 & 0.81001782 & 0.51734626 & 0.21562869 & 0.1919382 & -9751.2193 & 0.64887601 & -0.0006273 \\
23 & -0.7369246 & 0.46137829 & 0.26717875 & 0.81001726 & 0.51734702 & 0.21562892 & 0.19193779 & -9751.2193 & 0.64887685 & -0.000436\\ \hline
\end{tabular}
}
\vspace{-0.6cm}
\end{table*}

See \cite{nzokem2022fitting} for details on fitting GTS distribution with fast FRFT.

\begin{figure}[ht]
\vspace{-0.3cm}
    \centering
  \begin{subfigure}[b]{0.44\linewidth}
    \includegraphics[width=\linewidth]{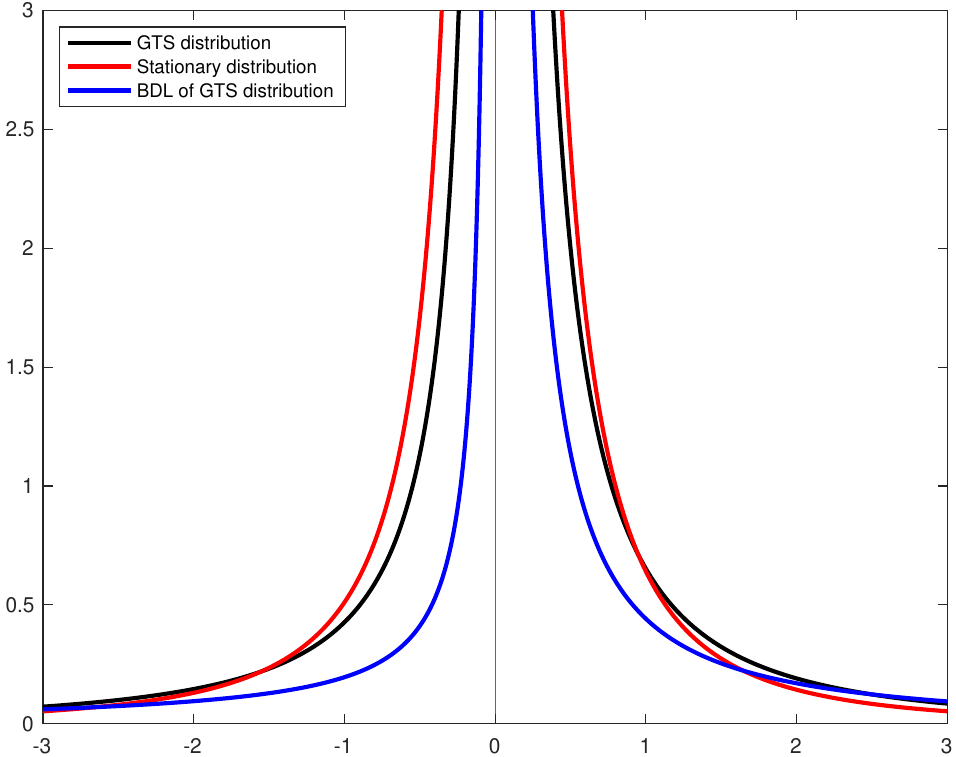}
\vspace{-0.4cm}
     \caption{L\'evy densities }
         \label{fig911}
  \end{subfigure}
  \begin{subfigure}[b]{0.5\linewidth}
    \includegraphics[width=\linewidth]{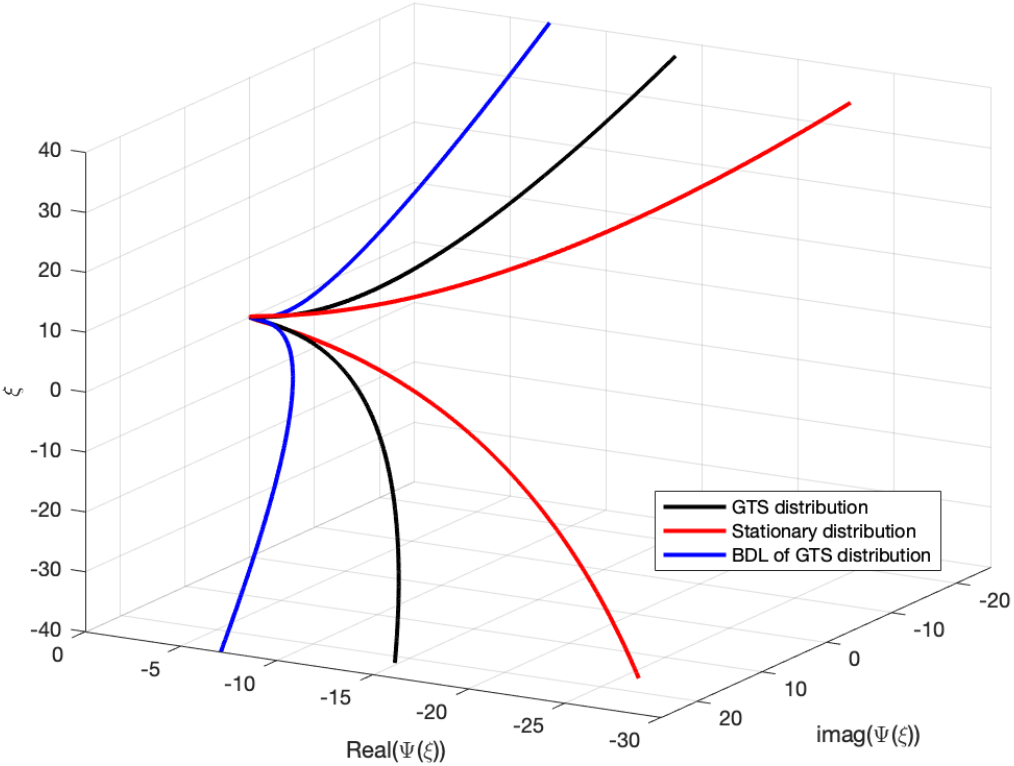}
\vspace{-0.6cm}
     \caption{Characteric exponents}
         \label{fig921}
          \end{subfigure}
\vspace{-0.6cm}
  \caption{Bitcoin Data: $\mu=-0.737459$, $\beta_{+}=0.461722$, $\beta_{-}=0.267500$, $\alpha_{+}=0.810017$, $\alpha_{-}=0.517386$, $\lambda_{+}=0.215545$, $\lambda_{-}=0.191874$}
  \label{fig01}
\vspace{-0.6cm}
\end{figure}

\newpage
\section { Applications and Simulation of S\&P index and Bitcoin process}

\subsection { Stationary process Algorithm of Ornstein-Uhlenbeck type}

We recall the stationary process of Ornstein-Uhlenbeck type $\left\{X(t)\right\}_{t>0}$ and a background driving L\'evy process (BDLP)  $\left\{L(t)\right\}_{t>0}$ independent of $X_{0}$ , such that $X(t)\overset{d}{=}X$ as described in Theorem (\ref{lem2})
\begin{align}
X_{t}=e^{-\lambda t} X_{0} + \int_{0}^{t}e^{-\lambda (t-s)} dL_{\lambda s}=e^{-\lambda t} X_{0} + \int_{0}^{\lambda t}e^{-\lambda t + s} dL_{s}     \quad  \lambda >0. 
\label {eq:l61}
  \end{align}
\begin{theorem} \label{lem8} \ \\
Let $\phi$ be the characteristic exponent of a random variable X  and $\Psi(x)$ be the characteristic exponent of a random variable  $L_{1}$. 
$\left\{X(t)\right\}_{t>0}$ is the stationary process of Ornstein-Uhlenbeck type, such that $X(t)\overset{d}{=}X$.
 \begin{align}
\lim_{t \to +\infty} Log (E\left[e^{i \xi X_{t}}\right])=\phi(\xi)=\int_{0}^{\xi}\frac{\Psi(u)}{u}du   
\label {eq:l62}
  \end{align}
\end{theorem} 
\begin{proof} \ \\
$X_{0}$ is an initial value, not a variable.
\noindent
\begin{equation*}
 \begin{aligned}
 Log (E\left[e^{i \xi X_{t}}\right]) &= e^{-\lambda t} X_{0} \xi i +  \int_{0}^{\lambda t}Log(E\left[e^{i \xi e^{-\lambda t + s} L_{1}}\right])ds=e^{-\lambda t} X_{0} \xi i +  \int_{0}^{\lambda t}\Psi( \xi e^{-\lambda t + s})ds\\
 &=e^{-\lambda t} X_{0} \xi i +  \int_{\xi e^{-\lambda t + s}}^{\xi}\frac{\Psi( u)}{u}du \quad \quad u=\xi e^{-\lambda t + s}
 \end{aligned}
\end{equation*} 

We take the limit on both side
\begin{equation*}
 \begin{aligned}
\lim_{t \to +\infty} Log (E\left[e^{i \xi X_{t}}\right])=\phi(\xi)=\int_{0}^{\xi}\frac{\Psi(u)}{u}du
 \end{aligned}
\end{equation*} 
\end{proof}

\begin{theorem} \label{lem9} \ \\
\noindent
Let us consider the stationary process of Ornstein-Uhlenbeck type $\left\{X(t)\right\}_{t>0}$ with  background driving L\'evy process (BDLP)  $\left\{L(t)\right\}_{t>0}$, such that $X(t)\overset{d}{=}X$. The discrete-time horizon is $t_{0}< t_{1}< t_{2}< t_{3}, . . . , < t_{n-1}< t_{n}$ with 
$\Delta t = t_{i}-t_{i-1}$. The discrete version of  the stationary process of Ornstein-Uhlenbeck type becomes
 \begin{align}
X_{k\Delta t}=  e^{-\lambda \Delta t}X_{(k-1)\Delta t} + \int_{0}^{\lambda \Delta t}e^{-\lambda \Delta t + s} dL_{s}\label {eq:l03}
  \end{align}
\end{theorem} 
\begin{proof} \ \\
For $k=1$, replace $t$ by $\Delta t$ in the Equation (\ref{eq:l10}) and we have :
\begin{equation*}
\begin{aligned}
X_{\Delta t}=e^{-\lambda \Delta t} X_{0} + \int_{0}^{\lambda \Delta t}e^{-\lambda \Delta t + s} dL_{s}
  \end{aligned}
\end{equation*} 
We assume the equation (\ref{eq:l03}) holds for k. We have the following relation for $k+1$.

\begin{equation*}
\begin{aligned}
X_{(k+1)\Delta t}&=e^{-\lambda (k+1)\Delta t} X_{0} + \int_{0}^{\lambda (k+1)\Delta t}e^{-\lambda (k+1)\Delta t + s} dL_{s} \\
&=e^{-\lambda \Delta t}\left(e^{-\lambda k\Delta t} X_{0} + \int_{0}^{\lambda k\Delta t}e^{-\lambda k\Delta t + s} dL_{s} \right)+ \int_{\lambda k\Delta t}^{\lambda (k+1)\Delta t}e^{-\lambda (k+1)\Delta t + s} dL_{s}\\
&=e^{-\lambda \Delta t}X_{k\Delta t}+ \int_{0}^{\lambda \Delta t}e^{-\lambda \Delta t + u} dL_{u + \lambda k\Delta t} \quad \quad  u=s - \lambda k\Delta t \quad dL_{u + \lambda k\Delta t}=dL_{u} \\
&=e^{-\lambda \Delta t}X_{k\Delta t}+ \int_{0}^{\lambda \Delta t}e^{-\lambda \Delta t + u} dL_{u } \label {eq:l10}
  \end{aligned}
\end{equation*} 
\end{proof} 

\noindent
The relation (\ref{eq:l10})  can be used to simulate $Y=\int_{0}^{\lambda\Delta t} e^{-\lambda\delta t +s} dL_{s}$,  given that $X_{(k-1)\Delta t} \overset{d}{=}X $ and  $X_{k \Delta t} \overset{d}{=}X $. Let $\phi$ be the characteristic exponent of a random variable X. From (\ref{eq:l10}), we have 

\begin{equation}
\begin{aligned}
Log (E\left[e^{i\xi X_{\Delta t}}\right])=Log (E\left[e^{i\xi {e^{-\lambda \Delta t}X_{(k-1)\Delta t}}}\right]) + Log (E\left[e^{i \xi Y}\right]) 
  \end{aligned}
\end{equation}
It is 
\begin{equation}
\begin{aligned}
\phi(\xi) =\phi(e^{-\lambda\Delta t} \xi) + Log (E\left[e^{i \xi Y}\right]) \quad \quad Log (E\left[e^{i \xi Y}\right])=\phi(\xi) - \phi({e^{-\lambda \Delta t}\xi}) \label {eq:l50}
   \end{aligned}
\end{equation}
\medskip
The Fourier transform generated by the characteristic exponent of $Y=\int_{0}^{\lambda\Delta t} e^{-\lambda\delta t +s} dL_{s}$ in (\ref{eq:l50}) and the probability density function (PDF) have the following expression.
\begin{equation}
F_{Y}[f](\xi) =   e^{\phi(-\xi) - \phi(-{e^{-\lambda \Delta t}\xi})}  \quad \quad f_{Y}(y) = \frac{1} {2\pi}\int_{-\infty}^{+\infty}\! F[f](\xi) e^{iy\xi}\, \mathrm{d}\xi . \label {eq:l51}
\end{equation}
\noindent
The composite of a FRFT of a 15-long weighted sequence and a FRFT of an N-long sequence was used to compute the probability density functions \ref{eq:l51}. See \cite{nzokem_2021,aubain2020, nzokem2023enhanced} for details on composite FRFTs and applications.
\begin{figure}[ht]
\vspace{-0.3cm}
    \centering
  \begin{subfigure}[b]{0.42\linewidth}
    \includegraphics[width=\linewidth]{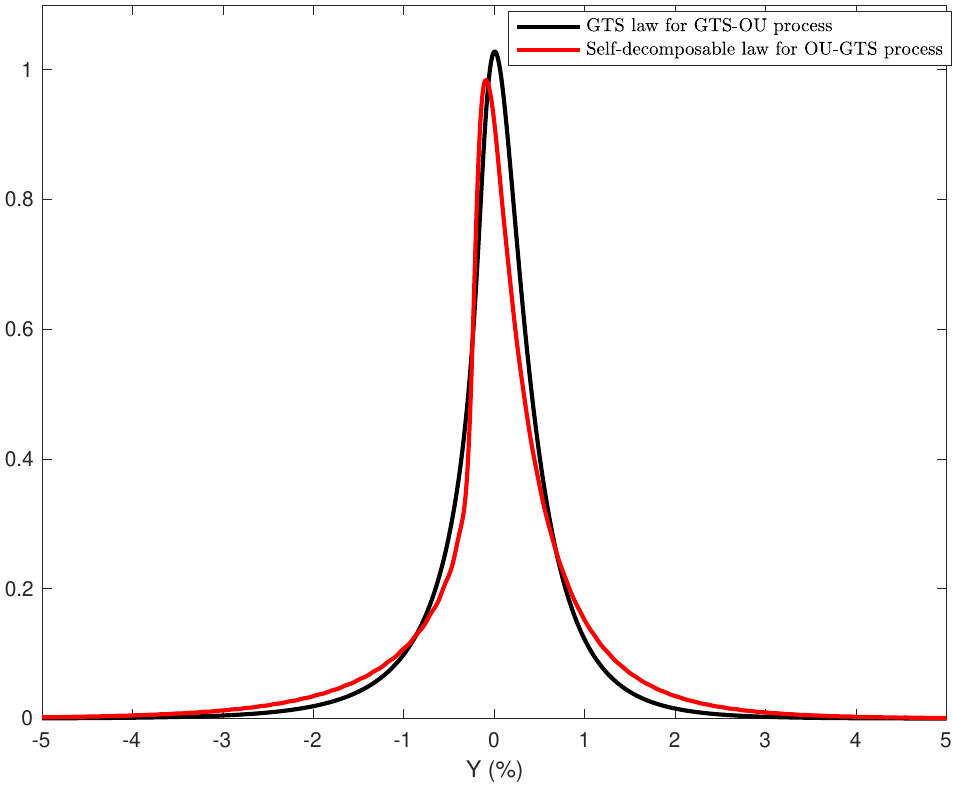}
\vspace{-0.6cm}
     \caption{S\&P500 index Daily Data }
         \label{fig911}
  \end{subfigure}
  \begin{subfigure}[b]{0.43\linewidth}
    \includegraphics[width=\linewidth]{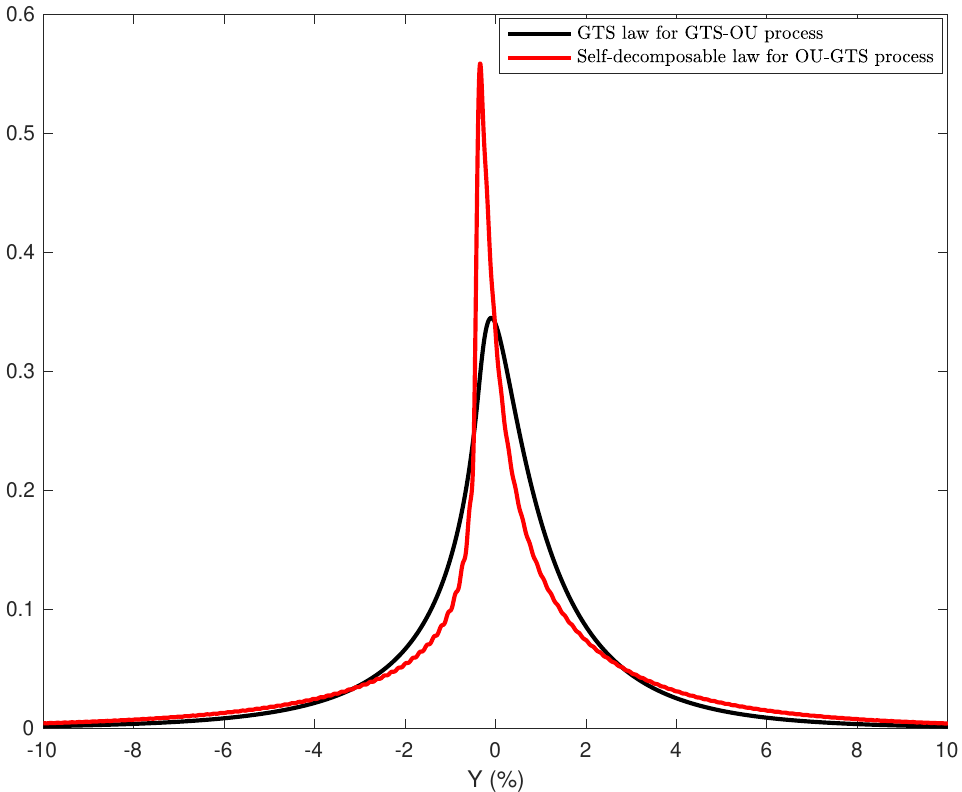}
\vspace{-0.6cm}
     \caption{Bitcoin Daily Data}
         \label{fig921}
          \end{subfigure}
\vspace{-0.6cm}
  \caption{ probability Density Function (PDF) of  $Y=\int_{0}^{\lambda\Delta t} e^{-\lambda\delta t +s} dL_{s}$ }
  \label{fig01}
\vspace{-0.6cm}
\end{figure}
\begin{algorithm}
\caption{Simulating the stationary process of Ornstein-Uhlenbeck type}
\begin{algorithmic}[1]
\Require GTS(\textbf{$\beta_{+}$}, \textbf{$\beta_{-}$}, \textbf{$\alpha_{+}$},\textbf{$\alpha_{-}$}, \textbf{$\lambda_{+}$}, \textbf{$\lambda_{-}$}); discrete-time $t_{0}< t_{1}< t_{2}, . . . , < t_{n-1}< t_{n}$
\State  Initialization: $X_{0} \gets x$.
\State For $i =1: N$
\State $a \gets e^{-\lambda \Delta t_{i}}$ \Comment{ $\Delta t_{i} = t_{i}-t_{i-1}$ and $\lambda=1$ }
\State $u_{i} \sim \mathscr{U}(0, 1)$ \Comment{ $ \mathscr{U}(0, 1)$ is is the Uniform distribution}
\State $y_{i}  \sim F_{Y}^{-1}(u_{i})$ \Comment{ $F_{Y}(y)$ is a CDF associated to the PDF $f_{Y}(y) $ in (\ref{eq:l51})}
\State $X_{i\Delta t_{i}}  \gets  a* X_{(i - 1)\Delta t_{i - 1}} +  y_{i} $
\State EndFor   
\end{algorithmic}
\end{algorithm}
\newpage
\subsection {Applications and Simulation results}
\noindent
The stationary process of Ornstein-Uhlenbeck type is performed using the previously developed Algorithm to simulate the S\&P index and Bitcoin process. The simulations are carried out on the GTS - OU  stationary process of the Ornstein-Uhlenbeck type and the OU - GTS stationary process of the Ornstein-Uhlenbeck type. The theoretical values generated by the the cumulant generating function  $(\ref{eq:l62})$ and the simulation values of means, variance, skewness, and kurtosis are compared; the relative error (Error \%) for each statistical indicator is also provided.\\
\noindent
The relation (\ref{eq:l62})  can be used as a cumulant function of the stationary process of Ornstein-Uhlenbeck type $\left\{X(t)\right\}_{t>0}$. The cumulants ($\kappa_{k}$) and the cumulant generating function of the GTS(\textbf{$\mu$}, \textbf{$\beta_{+}$}, \textbf{$\beta_{-}$}, \textbf{$\alpha_{+}$},\textbf{$\alpha_{-}$}, \textbf{$\lambda_{+}$}, \textbf{$\lambda_{-}$}) distribution was developed in Theorem 2.3 \cite{nzokem2023enhanced}.
\begin{equation}
\begin{aligned}
\Psi(\xi)=Log\left(Ee^{i Y\xi}\right)&= \sum_{j=1}^{+\infty}\kappa_{j}\frac{(i\xi)^{j}}{j!} \quad \kappa_{1}= \mu + \alpha_{+}{\frac{\Gamma(1-\beta_{+})}{\lambda_{+}^{1-\beta_{+}}}} - \alpha_{-}{\frac{\Gamma(1-\beta_{-})}{\lambda_{-}^{1-\beta_{-}}}} \\
 \kappa_{k} &=\alpha_{+}{\frac{\Gamma(k-\beta_{+})}{\lambda_{+}^{k-\beta_{+}}}} + (-1)^{k} \alpha_{-}{\frac{\Gamma(k-\beta_{-})}{\lambda_{-}^{k-\beta_{-}}}}\label {eq:l52}
  \end{aligned}
\end{equation}
where $ \kappa_{k}$ is the $k^{th}$ cumulant of the GTS distribution.\\
\noindent
The cumulant generating function of the stationary process of Ornstein-Uhlenbeck type $\left\{X(t)\right\}_{t>0}$  is deduced from relation (\ref{eq:l52})
\begin{equation}
\begin{aligned}
\phi(\xi)=\int_{0}^{\xi}\frac{\Psi(u)}{u}du =\sum_{k=1}^{+\infty}\frac{\kappa_{k}}{k}\frac{(i\xi)^{k}}{k!}  \quad m_{1}=\kappa_{1} \quad m_{k}=\frac{\kappa_{k}}{k}
\end{aligned}
\end{equation}
\noindent
We have the following relation for the mean, Variance, skewness, and Kurtosis for the stationary process of Ornstein-Uhlenbeck type $\left\{X(t)\right\}_{t>0}$.
\begin{equation}
\begin{aligned}
m_{1}&=\kappa_{1} = \mu + \alpha_{+}{\frac{\Gamma(1-\beta_{+})}{\lambda_{+}^{1-\beta_{+}}}} - \alpha_{-}{\frac{\Gamma(1-\beta_{-})}{\lambda_{-}^{1-\beta_{-}}}}\\
m_{2}&=\frac{\kappa_{2}}{2}=\frac{1}{2}\left[\alpha_{+}{\frac{\Gamma(k-\beta_{+})}{\lambda_{+}^{k-\beta_{+}}}} + (-1)^{k} \alpha_{-}{\frac{\Gamma(k-\beta_{-})}{\lambda_{-}^{k-\beta_{-}}}}\right]\\
\frac{m_{3}}{m_{2}^{\frac{3}{2}}}&=\frac{2^{\frac{3}{2}}}{3}\frac{\kappa_{3}}{\kappa_{2}^{\frac{3}{2}}} \quad \quad \left (3 + \frac{m_{4}}{m_{2}^{2}}\right)= \left (3 + \frac{\kappa_{4}}{\kappa_{2}^{2}}\right)\label {eq:l53}
\end{aligned}
\end{equation}
\subsubsection { Results for OU- GTS Process}
\noindent
The daily cumulative returns for the S\&P500 index and Bitcoin are simulated under the assumption that the daily return follows the stationary distribution defined by the L\'evy measure (\ref{eq:lem64}). The simulations are provided by the stationary process of Ornstein-Uhlenbeck type (\ref{eq:l03}). The sample path shows in Fig \ref{fig51} for the S\&P500 daily cumulative returns and in Fig \ref{fig52} for the Bitcoin daily cumulative returns. The sample path in red is compared to the actual values of daily cumulative returns in black.
\begin{figure}[ht]
\vspace{-0.3cm}
    \centering
\hspace{-0.6cm}
  \begin{subfigure}[b]{0.47\linewidth}
    \includegraphics[width=\linewidth]{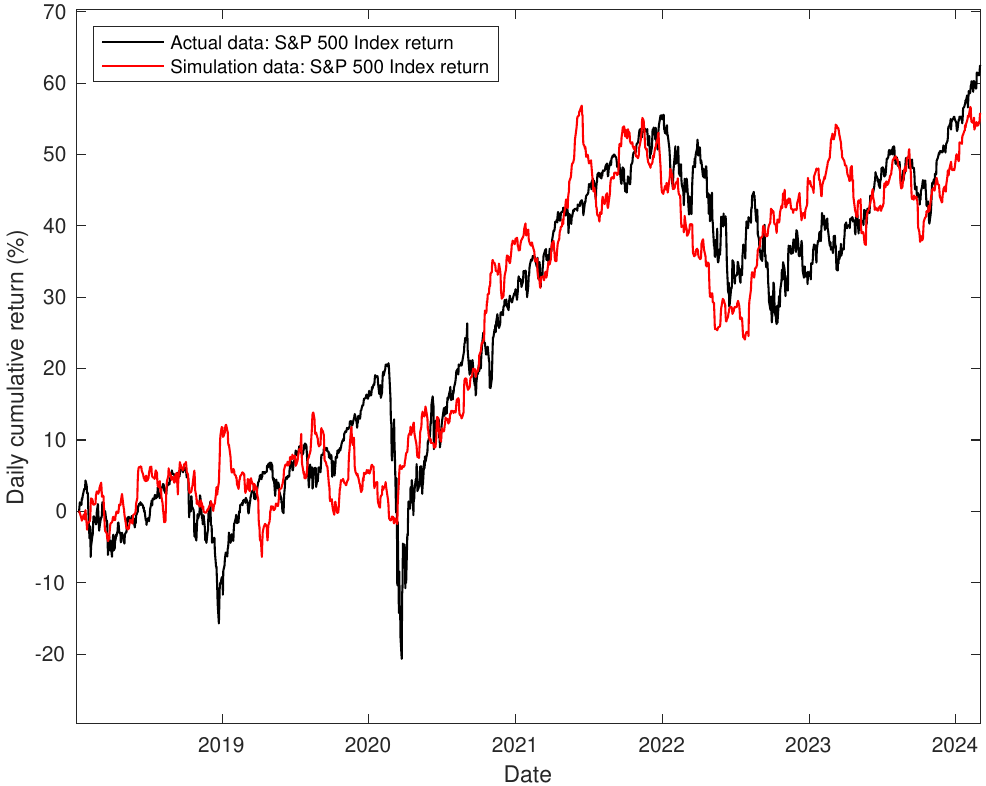}
\vspace{-0.5cm}
     \caption{ S\&P 500 index daily returns}
         \label{fig51}
  \end{subfigure}
\hspace{-0.3cm}
  \begin{subfigure}[b]{0.47\linewidth}
    \includegraphics[width=\linewidth]{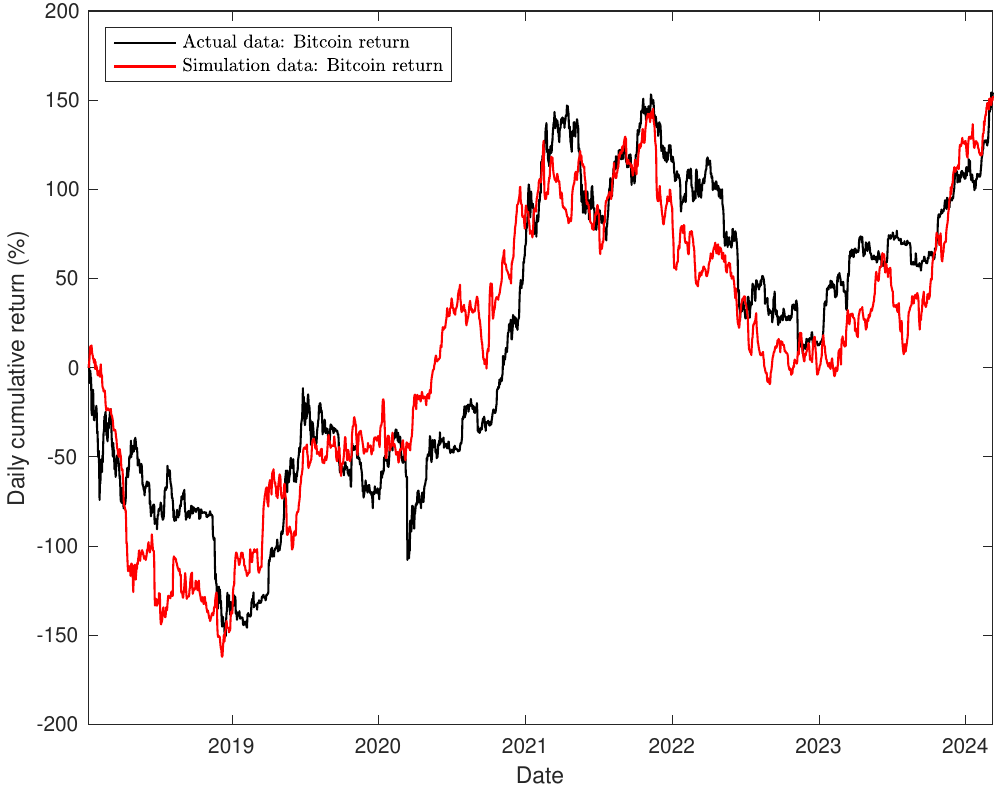}
\vspace{-0.5cm}
     \caption{Bitcoin daily returns}
         \label{fig52}
          \end{subfigure}
\vspace{-0.5cm}
 \caption{Simulation OU- GTS Process: daily cumulative returns ($\%$)}
  \label{fig53}
\vspace{-0.6cm}
\end{figure}

\noindent
The mean ($m_{1}$), Variance  ($m_{2}$), skewness and Kurtosis developed in (\ref{eq:l53}) for the stationary distribution were computed and summarised in Table (\ref{tab8}), on a raw labelled exact value. These theoretical values are compared with the empirical values from sample paths generated by the stationary process of Ornstein-Uhlenbeck type (\ref{eq:l03}). The sample size of the sample paths increases from 1000 to 5000. As shown in Table (\ref{tab8}), the empirical value converges to the exact value when the sample size increases for each statistical indicator. 

\begin{table}[ht]
\vspace{-0.3cm}
\centering
\caption{S$\&$P 500 return theoretical values versus the empirical values of the OU-GTS stationary process}
\vspace{-0.3cm}
\label{tab8}
\setlength\extrarowheight{1.3pt} 
\setlength\tabcolsep{1.3pt}  
\begin{tabular}{@{}c|cc|cc|cc|cc@{}}
\toprule
 & \multicolumn{2}{c}{$\displaystyle{\lim_{t \to \infty}}E\left[X_{t}|X_{0}\right]$} & \multicolumn{2}{|c|}{$\displaystyle{\lim_{t \to \infty}}Var\left[X^{2}_{t}|X_{0}\right]$} & \multicolumn{2}{c|}{$\displaystyle{\lim_{t \to \infty}}Skew \left[X_{t}|X_{0}\right]$} & \multicolumn{2}{c}{$\displaystyle{\lim_{t \to \infty}}Kurtosis \left[X_{t}|X_{0}\right]$} \\ \toprule
 Exact value& \multicolumn{2}{c|}{0.04013} & \multicolumn{2}{c|}{0.77410} & \multicolumn{2}{c|}{-0.54649} & \multicolumn{2}{c}{8.92320} \\
\toprule
\multicolumn{1}{c|}{\multirow{2}{*}{Sample Size}} & \multicolumn{8}{c}{Empirical Statistics from Simulations} \\ \cmidrule(l){2-9} 
\multicolumn{1}{c|}{} & \multicolumn{1}{c|}{} & \multicolumn{1}{c|}{Error(\%)} & \multicolumn{1}{c|}{} & \multicolumn{1}{c|}{Error(\%)} & \multicolumn{1}{c|}{} & \multicolumn{1}{c|}{Error(\%)} & \multicolumn{1}{c|}{} & \multicolumn{1}{c}{Error(\%)} \\ \toprule

\multirow{1}{*}{1000} & \multirow{1}{*}{0.03835} & \multirow{1}{*}{-4.45524} & \multirow{1}{*}{0.76772} & \multirow{1}{*}{-0.82393} & \multirow{1}{*}{-0.44509} & \multirow{1}{*}{-18.55559} & \multirow{1}{*}{7.53707} & -15.53397 \\ \midrule
\multirow{1}{*}{1500} & \multirow{1}{*}{0.03488} & \multirow{1}{*}{-13.09001} & \multirow{1}{*}{0.77405} & \multirow{1}{*}{-0.00658} & \multirow{1}{*}{-0.57732} & \multirow{1}{*}{5.64163} & \multirow{1}{*}{8.98251} & 0.66473 \\ \midrule
\multirow{1}{*}{2500} & \multirow{1}{*}{0.03956} & \multirow{1}{*}{-1.41919} & \multirow{1}{*}{0.77140} & \multirow{1}{*}{-0.34865} & \multirow{1}{*}{-0.53762} & \multirow{1}{*}{-1.62352} & \multirow{1}{*}{8.52938} & -4.41345 \\ \midrule
\multirow{1}{*}{5000} & \multirow{1}{*}{0.04039} & \multirow{1}{*}{0.64644} & \multirow{1}{*}{0.77349} & \multirow{1}{*}{-0.07943} & \multirow{1}{*}{-0.54792} & \multirow{1}{*}{0.26251} & \multirow{1}{*}{8.87311} & \multirow{1}{*}{-0.56136} \\ \bottomrule
\end{tabular}
\end{table}
\noindent
Table (\ref{tab8}) shows the results for S$\&$P 500 returns, whereas the results for Bitcoin returns are displayed in Table \ref{tab9} in Appendice \ref{eq:an2}. Both tables show that the simulation process is effective, and each sample path comes from the OU-GTS stationary process. 
\subsubsection { Results for GTS - OU Process}
\noindent
The daily cumulative returns for S\&P500 index and Bitcoin are simulated under the assumption that the daily return follows the GTS distribution. The sample path shows in Fig \ref{fig61} for S\&P500 daily cumulative returns and in Fig \ref{fig62} for Bitcoin daily cumulative returns. The sample path in red is compared to the actual values of daily cumulative returns in black.\begin{figure}[ht]
\vspace{-0.3cm}
    \centering
\hspace{-0.5cm}
  \begin{subfigure}[b]{0.47\linewidth}
    \includegraphics[width=\linewidth]{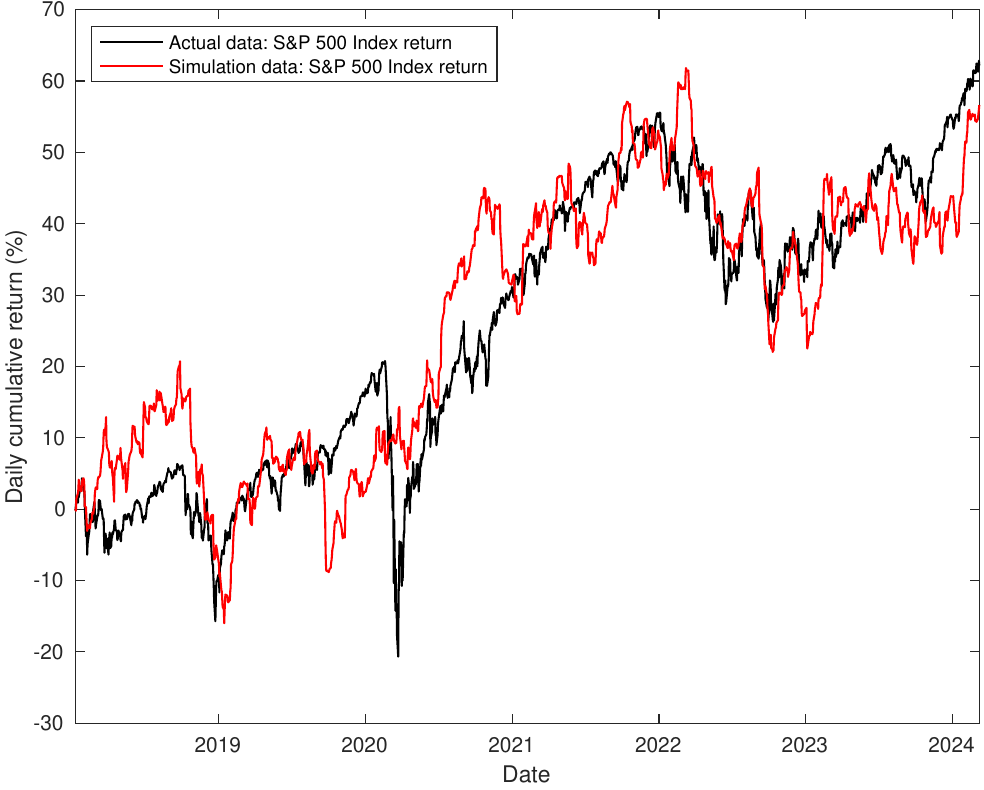}
\vspace{-0.5cm}
     \caption{ S\&P 500 index daily returns}
         \label{fig61}
  \end{subfigure}
\hspace{-0.3cm}
  \begin{subfigure}[b]{0.47\linewidth}
    \includegraphics[width=\linewidth]{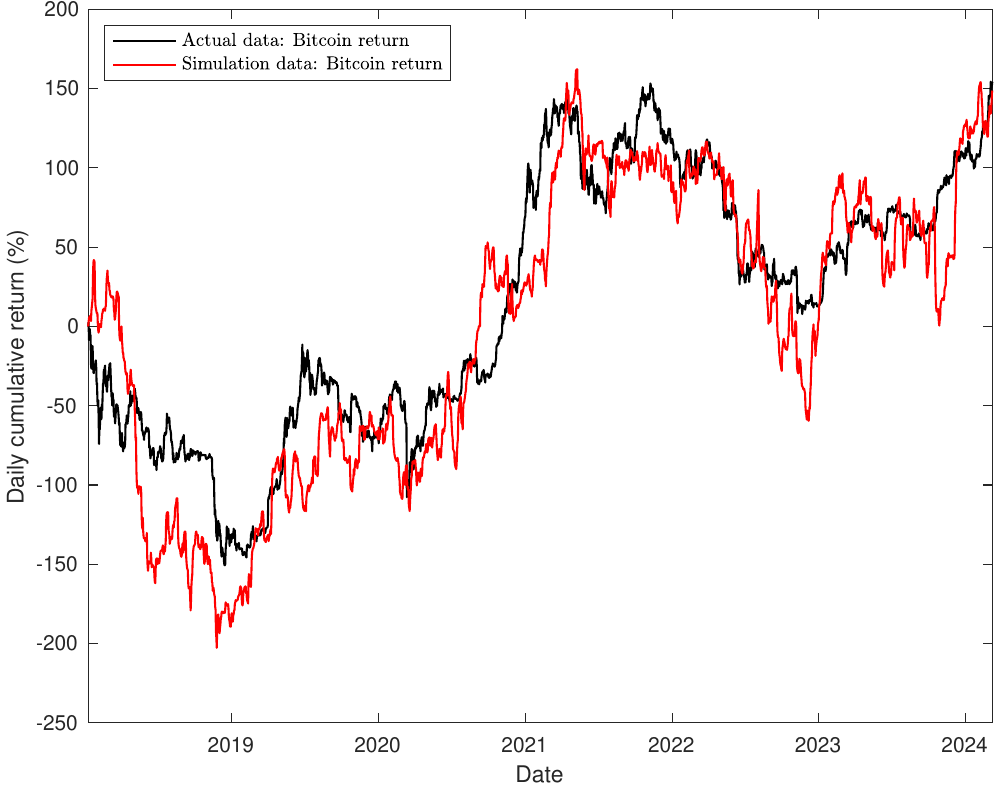}
\vspace{-0.5cm}
     \caption{Bitcoin daily returns}
         \label{fig62}
          \end{subfigure}
\vspace{-0.5cm}
 \caption{Simulation GTS - OU Process: Daily cumulative returns ($\%$)}
  \label{fig63}
\vspace{-0.6cm}
\end{figure}

\noindent
The mean ($m_{1}$) , Variance  ($m_{2}$), skewness and Kurtosis developed in (\ref{eq:l52}) for GTS distribution were computed and summarised in Table (\ref{tab10}). The empirical value converges to the exact value for each statistical indicator when the sample size increases.  
\begin{table}[ht]
\vspace{-0.3cm}
\centering
\caption{ S$\&$P 500 return theoretical values versus empirical values of the GTS - OU  stationary process}
\vspace{-0.3cm}
\label{tab10}
\setlength\extrarowheight{1.3pt} 
\setlength\tabcolsep{1.3pt}  
\begin{tabular}{@{}c|cc|cc|cc|cc@{}}
\toprule
 & \multicolumn{2}{c}{$\displaystyle{\lim_{t \to \infty}}E\left[X_{t}|X_{0}\right]$} & \multicolumn{2}{|c|}{$\displaystyle{\lim_{t \to \infty}}Var\left[X^{2}_{t}|X_{0}\right]$} & \multicolumn{2}{c|}{$\displaystyle{\lim_{t \to \infty}}Skew \left[X_{t}|X_{0}\right]$} & \multicolumn{2}{c}{$\displaystyle{\lim_{t \to \infty}}Kurtosis \left[X_{t}|X_{0}\right]$} \\ \toprule
\multirow{1}{*}{Exact value} & \multicolumn{2}{c|}{0.04013} & \multicolumn{2}{c|}{1.09475} & \multicolumn{2}{c|}{-0.57964} & \multicolumn{2}{c|}{8.92320} \\ \toprule
\multicolumn{1}{c|}{\multirow{2}{*}{Sample Size}} & \multicolumn{8}{c}{Empirical Statistics from Simulations} \\ \cmidrule(l){2-9} 
\multicolumn{1}{c|}{} & \multicolumn{1}{c|}{} & \multicolumn{1}{c|}{Error(\%)} & \multicolumn{1}{c|}{} & \multicolumn{1}{c|}{Error(\%)} & \multicolumn{1}{c|}{} & \multicolumn{1}{c|}{Error(\%)} & \multicolumn{1}{c|}{} & \multicolumn{1}{c}{Error(\%)} \\ \toprule

\multirow{1}{*}{1000} & \multirow{1}{*}{0.03682} & \multirow{1}{*}{-8.26002} & \multirow{1}{*}{1.08798} & \multirow{1}{*}{-0.61770} & \multirow{1}{*}{-0.56865} & \multirow{1}{*}{-1.89692} & \multirow{1}{*}{8.83489} & \multirow{1}{*}{-0.98962} \\ \midrule
\multirow{1}{*}{1500} & \multirow{1}{*}{0.04204} & \multirow{1}{*}{4.75799} & \multirow{1}{*}{1.10069} & \multirow{1}{*}{0.54257} & \multirow{1}{*}{-0.57351} & \multirow{1}{*}{-1.05806} & \multirow{1}{*}{8.56435} & \multirow{1}{*}{-4.02145} \\ \midrule
\multirow{1}{*}{2500} & \multirow{1}{*}{0.04224} & \multirow{1}{*}{5.23894} & \multirow{1}{*}{1.09164} & \multirow{1}{*}{-0.28379} & \multirow{1}{*}{-0.53099} & \multirow{1}{*}{-8.39354} & \multirow{1}{*}{8.64516} & \multirow{1}{*}{-3.11584} \\ \midrule
\multirow{1}{*}{5000} & \multirow{1}{*}{0.04178} & \multirow{1}{*}{4.11213} & \multirow{1}{*}{1.09141} & \multirow{1}{*}{-0.30449} & \multirow{1}{*}{-0.56619} & \multirow{1}{*}{-2.32057} & \multirow{1}{*}{8.70841} & \multirow{1}{*}{-2.40708} \\ \bottomrule
\end{tabular}%
\end{table}

\noindent
Table (\ref{tab10}) shows the results for S$\&$P 500 Returns, whereas the results for Bitcoin returns are displayed in Table \ref{tab11} in Appendice \ref{eq:an2}. Both tables show that the simulation process is effective, and each sample path comes from the GTS-OU stationary process.  

\begin{appendices}
\section{Bitcoin: Theoretical versus Empirical values }\label{eq:an2}
\begin{table}[ht]
\vspace{-0.3cm}
\centering
\caption{Bitcoin return theoretical values versus empirical values of the OU-GTS stationary process}
\vspace{-0.3cm}
\label{tab9}
\setlength\extrarowheight{1.3pt} 
\setlength\tabcolsep{1.3pt}  
\begin{tabular}{@{}c|cc|cc|cc|cc@{}}
\toprule
 & \multicolumn{2}{c}{$\displaystyle{\lim_{t \to \infty}}E\left[X_{t}|X_{0}\right]$} & \multicolumn{2}{|c|}{$\displaystyle{\lim_{t \to \infty}}Var\left[X^{2}_{t}|X_{0}\right]$} & \multicolumn{2}{c|}{$\displaystyle{\lim_{t \to \infty}}Skew \left[X_{t}|X_{0}\right]$} & \multicolumn{2}{c}{$\displaystyle{\lim_{t \to \infty}}Kurtosis \left[X_{t}|X_{0}\right]$} \\ \toprule
\multirow{1}{*}{Exact value} & \multicolumn{2}{c|}{0.14890} & \multicolumn{2}{c|}{2.81898} & \multicolumn{2}{c|}{-0.30158} & \multicolumn{2}{c}{9.74634} \\ \toprule
\multicolumn{1}{c|}{\multirow{2}{*}{Sample Size}} & \multicolumn{8}{c}{Empirical Statistics from Simulations} \\ \cmidrule(l){2-9} 
\multicolumn{1}{c|}{} & \multicolumn{1}{c|}{} & \multicolumn{1}{c|}{Error(\%)} & \multicolumn{1}{c|}{} & \multicolumn{1}{c|}{Error(\%)} & \multicolumn{1}{c|}{} & \multicolumn{1}{c|}{Error(\%)} & \multicolumn{1}{c|}{} & \multicolumn{1}{c}{Error(\%)} \\ \toprule
\multirow{1}{*}{1000} & \multirow{1}{*}{0.13623} & \multirow{1}{*}{-8.51088} & \multirow{1}{*}{2.81760} & \multirow{1}{*}{-0.04900} & \multirow{1}{*}{-0.30070} & \multirow{1}{*}{-0.29135} & \multirow{1}{*}{8.23142} & \multirow{1}{*}{-15.54341} \\ \midrule
\multirow{1}{*}{1500} & \multirow{1}{*}{0.13806} & \multirow{1}{*}{-7.28426} & \multirow{1}{*}{2.80073} & \multirow{1}{*}{-0.64725} & \multirow{1}{*}{-0.32346} & \multirow{1}{*}{7.25614} & \multirow{1}{*}{9.37698} & \multirow{1}{*}{-3.78972} \\ \midrule
\multirow{1}{*}{2500} & \multirow{1}{*}{0.15173} & \multirow{1}{*}{1.89851} & \multirow{1}{*}{2.81712} & \multirow{1}{*}{-0.06598} & \multirow{1}{*}{-0.29138} & \multirow{1}{*}{-3.38203} & \multirow{1}{*}{9.57965} & \multirow{1}{*}{-1.71027} \\ \midrule
\multirow{1}{*}{5000} & \multirow{1}{*}{0.14987} & \multirow{1}{*}{0.64693} & \multirow{1}{*}{2.81808} & \multirow{1}{*}{-0.03186} & \multirow{1}{*}{-0.29731} & \multirow{1}{*}{-1.41514} & \multirow{1}{*}{9.66934} & \multirow{1}{*}{-0.79007} \\ \bottomrule
\end{tabular}%
\end{table}

\begin{table}[ht]
\vspace{-0.3cm}
\centering
\caption{Bitcoin return theoretical values versus empirical values of the GTS - OU  stationary process}
\vspace{-0.3cm}
\label{tab11}
\setlength\extrarowheight{1.3pt} 
\setlength\tabcolsep{1.3pt}  
\begin{tabular}{@{}c|cc|cc|cc|cc@{}}
\toprule
 & \multicolumn{2}{c}{$\displaystyle{\lim_{t \to \infty}}E\left[X_{t}|X_{0}\right]$} & \multicolumn{2}{|c|}{$\displaystyle{\lim_{t \to \infty}}Var\left[X^{2}_{t}|X_{0}\right]$} & \multicolumn{2}{c|}{$\displaystyle{\lim_{t \to \infty}}Skew \left[X_{t}|X_{0}\right]$} & \multicolumn{2}{c}{$\displaystyle{\lim_{t \to \infty}}Kurtosis \left[X_{t}|X_{0}\right]$} \\ \toprule
\multirow{1}{*}{Exact value}  & \multicolumn{2}{c|}{0.14890} & \multicolumn{2}{c|}{3.98664} & \multicolumn{2}{c|}{-0.31987} & \multicolumn{2}{c|}{9.74634} \\ \toprule
\multicolumn{1}{c|}{\multirow{2}{*}{Sample Size}} & \multicolumn{8}{c}{Empirical Statistics from Simulations} \\ \cmidrule(l){2-9} 
\multicolumn{1}{c|}{} & \multicolumn{1}{c|}{} & \multicolumn{1}{c|}{Error(\%)} & \multicolumn{1}{c|}{} & \multicolumn{1}{c|}{Error(\%)} & \multicolumn{1}{c|}{} & \multicolumn{1}{c|}{Error(\%)} & \multicolumn{1}{c|}{} & \multicolumn{1}{c}{Error(\%)} \\ \toprule
\multirow{1}{*}{1000} & \multirow{1}{*}{0.15521} & \multirow{1}{*}{4.23755} & \multirow{1}{*}{4.00641} & \multirow{1}{*}{0.49589} & \multirow{1}{*}{-0.24091} & \multirow{1}{*}{-24.68461} & \multirow{1}{*}{9.49887} & -2.53906 \\ \midrule
\multirow{1}{*}{1500} & \multirow{1}{*}{0.11793} & \multirow{1}{*}{-20.80070} & \multirow{1}{*}{3.98615} & \multirow{1}{*}{-0.01217} & \multirow{1}{*}{-0.36310} & \multirow{1}{*}{13.51377} & \multirow{1}{*}{9.56633} & -1.84691 \\ \midrule
\multirow{1}{*}{2500} & \multirow{1}{*}{0.14404} & \multirow{1}{*}{-3.26698} & \multirow{1}{*}{3.96137} & \multirow{1}{*}{-0.63372} & \multirow{1}{*}{-0.33935} & \multirow{1}{*}{6.08765} & \multirow{1}{*}{9.22822} & -5.31603 \\ \midrule
\multirow{1}{*}{5000} & \multirow{1}{*}{0.15275} & \multirow{1}{*}{2.58583} & \multirow{1}{*}{3.97762} & \multirow{1}{*}{-0.22608} & \multirow{1}{*}{-0.32867} & \multirow{1}{*}{2.75017} & \multirow{1}{*}{9.42721} & \multirow{1}{*}{-3.27430} \\ \bottomrule
\end{tabular}%
\end{table}
\end{appendices}
\newpage
\bibliography{fourierbis}
\end{document}